\documentclass[numbers,webpdf]{ima-authoring-template}%
\graphicspath{{Fig/}}

\theoremstyle{thmstyletwo}
\newtheorem{theorem}{Theorem}
\newtheorem{proposition}[theorem]{Proposition}

\newtheorem{remark}{Remark}
\newtheorem{definition}{Definition}
\newtheorem{lemma}[theorem]{Lemma}
\newtheorem{problem}{Problem}
\numberwithin{equation}{section}

\usepackage{hyperref}

\allowdisplaybreaks[4]

\DeclareMathOperator*{\arginf}{arg\,inf}
\newcommand{\rd}{\mathrm{d}}
\newcommand{\R}{\mathbb{R}}
\newcommand{\T}{\mathsf{T}}

\begin{document}

\DOI{DOI HERE}
\copyrightyear{2021}
\vol{00}
\pubyear{2021}
\access{Advance Access Publication Date: Day Month Year}
\appnotes{Paper}
\copyrightstatement{Published by Oxford University Press on behalf of the Institute of Mathematics and its Applications. All rights reserved.}
\firstpage{1}

\title[JKO for Landau]{JKO for Landau: a variational particle method for homogeneous Landau equation}

\author{Yan Huang\ORCID{0009-0002-9077-9929} and Li Wang*\ORCID{0000-0002-0593-8175}
\address{\orgdiv{School of Mathematics}, \orgname{University of Minnesota}, \orgaddress{\street{206 Church St. SE}, \street{Minneapolis}, \postcode{55455}, \state{MN}, \country{USA}}}}

\corresp[*]{Corresponding author: \href{liwang@umn.edu}{liwang@umn.edu}}

\received{Date}{0}{Year}
\revised{Date}{0}{Year}
\accepted{Date}{0}{Year}

\abstract{
Inspired by the gradient flow viewpoint of the Landau equation and the corresponding dynamic formulation of the Landau metric in \cite{carrillo2024landau}, we develop a novel implicit particle method for the Landau equation in the framework of the JKO scheme. We first reformulate the Landau metric in a computationally friendly form, and then translate it into the Lagrangian viewpoint using the flow map. A key observation is that, while the flow map evolves according to a rather complicated integral equation, the unknown component is simply a score function of the corresponding density plus an additional term in the null space of the collision kernel. This insight guides us in designing and training the neural network for the flow map. Additionally, the objective function is in a double summation form, making it highly suitable for stochastic methods. Consequently, we design a tailored version of stochastic gradient descent that maintains particle interactions and significantly reduces the computational complexity. Compared to other deterministic particle methods, the proposed method enjoys exact entropy dissipation and unconditional stability, therefore making it suitable for large-scale plasma simulations over extended time periods.}


\maketitle

\section{Introduction} 
The Landau equation is a fundamental kinetic equation used to model the behavior of charged particles interacting via Coulomb forces \cite{landau1958kinetic}. It is especially relevant for plasmas where collision effects can be significant. When spatial dependence is ignored, the Landau equation takes the following form:
\begin{equation}\label{landau}
    \partial_t f = Q(f,f) =: \nabla_v \cdot \left[ \int_{\R^{d}} C_{\gamma} A(v-v_*) \left( f(t, v_*)\nabla_v f(t, v) - f(t, v)\nabla_{v_*} f(t, v_*) \right) \rd v_* \right] \,,
\end{equation}
where $f(t, v)$ for $(t, v) \in \R^{+}  \times \R^d$ with $d \geq 2$, is the mass distribution function of charged particles, such as electrons or ions, at time $t$ with velocity $v$. The operator $Q(f, f)$ is known as the Landau collision operator, which can be derived from the Boltzmann collision operator when small-angle deviations dominate. The collision kernel $A$ is given by:
\begin{equation*}
    A(z)=|z|^{\gamma+2} \left(I_d - \frac{z \otimes z}{|z|^{2}} \right) =: |z|^{\gamma+2} \Pi(z) \,, 
\end{equation*}
where $I_d$ is the identity matrix. As written, $\Pi(z)$ denotes the projection onto $\{z\}^{\perp}$. The parameter $C_{\gamma} > 0$ is the collision strength. The exponent $\gamma$ determines the type of interaction which ranges from $-d-1$ to $1$. Specifically, $0 < \gamma \leq 1$ is associated with hard potentials, while $\gamma < 0$ corresponds to the soft potentials. Of particular interest is the case when $d = 3$ and $\gamma = -3$, which corresponds to the Coulomb interaction in plasma \cite{degond1992fokker, Villani1998b}. Alternatively, $\gamma=0$ is known as the Maxwellian case, where the equation simplifies to a degenerate linear Fokker-Planck equation \cite{villani1998a}.

The theoretical understanding of \eqref{landau} remains limited and continues to be an active area of research. The well-established case is for hard potentials or the Maxwellian molecules, with seminal works \cite{villani2000a, villani2000b, villani1998a} and related literature addressing both the well-posedness and regularity of the solution. Significant progress on soft potentials was made in \cite{Guo2002}, which covers the global existence and uniqueness of solutions, but only for those near the Maxwellian distribution. A recent breakthrough in \cite{guillen2023landau} tackles the Coulomb case, showing that the Fisher information is monotonically decreasing over time, thereby ensuring that the solution remains globally bounded in time.

The numerical computation of \eqref{landau} also presents significant challenges due to the complexity of the operator $Q$, the high dimensionality in $v$, the stiffness with large collision strength $C_\gamma$, and the need to preserve the physical properties of the solution. Previous efforts have included methods such as the spectral method \cite{Pareschi_Russo_Toscani_2000, ZHANG2017470, Li_Wang_Wang_2020}, the stochastic Monte Carlo method \cite{dimarco2010direct, ROSIN2014140}, the finite difference entropy method \cite{degond1994entropy, buet1998conservative}, the deterministic particle method \cite{carrillo2020particle, Hirvijoki_2021}, and the time implicit methods \cite{Lemou2005, TAITANO2015357}, each addressing various aspects of these challenges. 

In this paper, we aim to develop a method that simultaneously addresses all the challenges except for exact energy conservation, which is maintained only up to $O(\Delta t)$ accuracy. Our approach builds upon the novel gradient flow perspective of \eqref{landau} introduced in \cite{carrillo2024landau}, ensuring that it is structure-preserving by design. To handle high dimensions more efficiently, our method uses sampling-based particles and incorporates the approximation capabilities of neural networks. Unlike many physics-based machine learning approaches for solving PDEs, our use of neural networks is minimal, therefore significantly simplifying the training process. Compared to recent efforts that also employ neural networks for solving the Landau equation \cite{HUANG2025114053,ilin2024}, our method differs primarily in its design of an implicit scheme rather than an explicit one. As a result, our method is exactly entropy-dissipative and exhibits significantly improved stability for stiff problems, making it highly suitable for large-scale time simulations of plasma. We also emphasize that, although our approach appears more involved than \cite{HUANG2025114053,ilin2024}, the computational complexity remains $O(N)$; the prefactor, however, may be larger due to the increased complexity of training. Nevertheless, we demonstrate through a concrete example in Section \ref{sec:5.3} that, despite the higher per-step cost, the overall efficiency gain for stiff problems makes the method worthwhile.

To lay out the main idea of our method, we briefly revisit the derivation of the Landau equation as a gradient flow \cite{carrillo2024landau}, in a formal manner. First, rewrite $Q$ in the log form 
\begin{equation*}
    Q(f, f) = \nabla \cdot \left[ \int_{\R^d} C_{\gamma} A(v-v_*)(\nabla\log f - \nabla_* \log f_{*}) f f_{*} \rd v_* \right] \,,
\end{equation*}
where the following abbreviated notations are used: 
\begin{equation*}
    f:=f(t, v), ~ 
    f_{*}:=f(t, v_*), ~
    \nabla := \nabla_{v}, ~
    \nabla_* := \nabla_{v_*}\,.
\end{equation*}
For an appropriate test function $\varphi = \varphi(v)$,  \eqref{landau} admits the following weak form:
\begin{equation}\label{landau2}
    \frac{\rd}{\rd t} \int_{\R^d} \varphi f \rd v = -\frac{1}{2} \iint_{\R^{2d}} C_{\gamma} (\nabla\varphi - \nabla_*\varphi_*) \cdot A(v-v_*) (\nabla\log f - \nabla_*\log f_{*}) f f_{*} \rd v \rd v_* \,.
\end{equation}
Choosing $\varphi(v) = 1, v, |v|^2$ leads to the conservation of mass, momentum, and energy. Inserting $\varphi(v) = \log f(v)$, one obtains the entropy decay due to the fact that $A$ is symmetric and semi-positive definite.

Define a new gradient operator:
\begin{equation}\label{new_grad}
    \tilde{\nabla} \varphi:= |v-v_*|^{1+\gamma/2} \Pi(v-v_*)(\nabla\varphi-\nabla_*\varphi_*) \,.
\end{equation}
Since $\Pi^2 = \Pi$, the weak form \eqref{landau2} can be rewritten into 
\begin{equation} \label{0910}
    \frac{\rd}{\rd t} \int_{\R^d} \varphi f \rd v = -\frac{1}{2} \iint_{\R^{2d}} \tilde{\nabla}\varphi \cdot \tilde{\nabla} \frac{\delta (C_{\gamma} H)}{\delta f} f f_{*} \rd v \rd v_* \,, ~ H = \int_{\R^d} f \log f \rd v \,,
\end{equation}
where $\tilde{\nabla} \cdot$ is the corresponding divergence operator in the distributional sense. More specifically, given a test function $\varphi(v) \in C_c^\infty(\R^d)$ and a vector-valued function $\psi=\psi(v,v_*) \in \R^d$, we have:
\begin{equation*}
    \iint_{\R^{2d}} \tilde{\nabla}\varphi(v,v_*) \cdot \psi(v,v_*) \rd v_* \rd v = - \int_{\R^d} \varphi(v) (\tilde{\nabla} \cdot \psi)(v)\rd v \,.
\end{equation*}
Equipped with this notation, one can rewrite the Landau equation \eqref{landau} as
\begin{equation}\label{landau3}
    \partial_t f = \frac{1}{2} \tilde{\nabla} \cdot \left(f f_{*} \tilde{\nabla} \frac{\delta (C_{\gamma} H)}{\delta f} \right) \,.
\end{equation}

Now drawing an analogy to the Wasserstein gradient flow perspective on the heat equation, \cite{carrillo2024landau} defines the following Landau metric between two probability densities $f_0$ and $f_1$ \footnote{In this paper, we do not distinguish between the notation for probability measures and densities.}, denoted as $d_L$:
\begin{equation} \label{dL}
    \begin{split}
        & d_L^2(f_0, f_1) :=~ \inf_{f, V} ~ \left\{ \int_0^1 \frac12 \iint_{\R^{2d}} |V|^2 ff_{*} \rd v \rd v_* \rd t \right\}, \\
        & s.t.~~ \partial_t f + \frac12 \tilde{\nabla} \cdot (V ff_{*}) = 0 \,, \ f(0, \cdot)=f_0 \,, \ f(1, \cdot)=f_1 \,.
    \end{split}
\end{equation}
As a result, the Landau equation \eqref{landau} can be viewed as the gradient flow of entropy $C_{\gamma} H$ with respect to the metric $d_L$. In particular, one can construct the weak solution of \eqref{landau} by De Giorgi's minimizing movement scheme \cite{Ambrosio2005GradientFI}, commonly known as the Jordan-Kinderlehrer-Otto (JKO) scheme \cite{jko}. Specifically, fix a time step $\Delta t > 0$, one recursively defines a sequence $\{f^n\}_{n=0}^{\infty}$ as
\begin{equation}\label{jko}
    f^0=f(0, \cdot) \,, ~ f^{n+1} \in \arginf_{f} \left[ d_L^2(f, f^n) + 2\Delta t C_{\gamma} H(f) \right] \,.
\end{equation}
The resulting solution will be a time-discrete approximation of \eqref{landau}, and with appropriate time interpolation, it will converge to the solution of \eqref{landau} as $\Delta t \rightarrow 0$ under certain conditions. More rigorous statements will be given in Section~\ref{sec-GF prospective}. 

Our numerical scheme is based on the formulations in \eqref{jko} and \eqref{dL}. However, although \eqref{dL} resembles the Benamou-Brenier dynamic formulation of the 2-Wasserstein metric \cite{Benamou2000}, it is not easily accessible from a computational standpoint. Therefore, we introduce the following alternative formulation of $d_L$:
\begin{equation}\label{dL-new}
    \begin{split}
        & d_L^2(f_0, f_1) :=~  \inf_{f, s} ~  \int_0^1 \frac{1}{2} \iint_{\R^{2d}} |s-s_*|^2_{A(v-v_*)} ff_{*} \rd v \rd v_* \rd t \,, \\ 
        & s.t. ~~ \partial_t f = \nabla \cdot \left[f \left( \int_{\R^{d}} A(v-v_*) (s-s_*)f_{*} \rd v_* \right)\right] \,,~ f(0, \cdot)=f_0 \,,~ f(1, \cdot)=f_1 \,,
    \end{split}
\end{equation}
such that the complex expression of $\tilde \nabla$ is explicitly spelled out. A more detailed relation between \eqref{dL} and \eqref{dL-new} will be outlined in Section~\ref{sec-GF prospective}. 

A key advantage of the new form \eqref{dL-new} is that both the objective function and the constrained PDE can be represented using particles. This is significant because particle methods, unlike grid-based methods, tend to scale better with dimensionality. Specifically, the objective function, which involves a double integral, can be interpreted as a double expectation with respect to the probability densities $f$ and $f_*$. This allows it to be approximated by the empirical sum of particles. Similarly, the constrained PDE can be viewed as a transport equation with an integral form of the velocity field, which can also be interpreted as an expectation and thus approximated using particles.

To update the particle velocities sequentially in time, we use the flow map representation, as previously adopted in \cite{LEE2024113187}. This approach has two favorable traits: First, it propagates both the particles and the density simultaneously, completely avoiding the challenges associated with density estimation, which is often a bottleneck in particle-based methods. Second, it transforms the constrained optimization problem \eqref{dL-new} into an unconstrained one, allowing optimization with respect to the flow map (or more specifically $s$ that generates the flow map) instead of both $f$ and $s$. 
This broadens the range of available optimization solvers and simplifies the optimization process.

In practice, $s$ is approximated by a neural network, which offers greater flexibility and is less sensitive to dimensionality compared to other approximations such as polynomials or Fourier series. Additionally, due to the particle representation in \eqref{dL-new}, the training process is well-suited to stochastic methods. Specifically, we design a tailored mini-batch stochastic gradient descent (SGD), inspired by the standard mini-batch SGD \cite{Bottou2018}. We also provide corresponding convergence analysis under common optimization theory assumptions. Once $s$ is trained, the particle update can be performed with improved efficiency using the random batch method \cite{Carrillo_Jin_Tang_2022}.

The rest of the paper is organized as follows. In the next section, we present key components in the design of our method. We begin with a review of the gradient flow perspective of the Landau equation and the dynamic formulation of the Landau metric. We then convert the Landau metric into a computable form and reformulate it using Lagrangian coordinates. This new metric, combined with the JKO scheme, forms the foundation of our variational approach. In Section \ref{sec3}, we detail our scheme, which includes a particle method and a neural network approximation. Since the objective function for training the neural network is in the form of a double summation, Section \ref{sec4} focuses on the stochastic optimization method and its convergence. Extensive numerical tests are presented in Section \ref{sec5}.

%%%%%%%%%%%%%%%%%%%%%%%%%%%%%%%%%%%%%%%%%
\section{Variational formulation and the JKO scheme}
This section outlines the foundational elements of our method. We first summarize the theoretical results from \cite{carrillo2024landau}, which establish the basis for viewing the Landau equation as a gradient flow. We also derive a more computationally friendly version of the Landau metric. This gradient flow perspective will then be translated into a dynamic JKO scheme, where we verify, via the optimality conditions, that the Landau equation is recovered to first order in time. Finally, we introduce a Lagrangian formulation in preparation for spatial discretization.

%%%%%%%%%%%%%%%%%%%%%%%%%%%
\subsection{Gradient flow perspective} \label{sec-GF prospective}
For simplicity, we assume the collision strength $C_\gamma = 1$ in this subsection. Let $\mathcal{P}_2(\R^d)$ be the space of probability measures with finite 2-moments:
$\mathcal{P}_{2}(\R^d) := \big\{ f: f \in \mathcal{P}(\R^d) \,, m_2(f) := \int_{\R^d} |v|^2 \rd f < \infty \big\}$. When the second moment is bounded by $E$, we denote this space by $\mathcal{P}_{2,E}(\R^d) \subset \mathcal{P}_2(\R^d)$. Following Villani's H-solutions \cite{Villani1998b}, the notion of a weak solution to \eqref{landau} is defined as:
\begin{definition}\cite[Definition 1]{carrillo2024landau}
    For $T>0$, we say that $f \in C([0,T]; L^1(\R^d))$ is a weak solution to the Landau equation \eqref{landau} if for every test function $\varphi \in C^\infty_c((0,T) \times \R^d)$, we have 
    $ \int_0^T \int_{\R^d} \partial_t \varphi f \rd v \rd t = \frac{1}{2} \int_0^T \iint_{\R^{2d}} \tilde{\nabla} \varphi \cdot \tilde{\nabla} \frac{\delta H}{\delta f}  f f_{*} \rd v \rd v_* \rd t$, and $f$ satisfies the following properties: 
    \begin{enumerate}[1.]
        \item $f$ is a probability measure with uniformly bounded second moment;
        \item The entropy $H(f) := \int_{\R^d} f \log f \rd v$ is bounded by its initial value: 
        $H(f(t, \cdot)) \leq H(f(0, \cdot)) < +\infty$ for all $t \in [0,T]$;
        \item The entropy dissipation is time integrable:
        \begin{equation}\label{entropy_diss}
            D_{H}(f) := \frac{1}{2} \iint_{\R^{2d}} \left| \tilde{\nabla} \frac{\delta H}{\delta f} \right|^2 f f_* \rd v \rd v_* \in L^1_t([0,T]) \,,
        \end{equation}
    \end{enumerate}
    where operator $\tilde{\nabla}$ is defined in \eqref{new_grad}.
\end{definition}
One main theorem in \cite{carrillo2024landau} establishes the equivalence between the gradient flow solution (defined as the curves of maximal slope, see Definition~\ref{def-maxslope}) and the weak solution of the Landau equation.

\begin{theorem}[Landau equation as a gradient flow]\cite[Theorem 12]{carrillo2024landau} 
    Fix $d=3$ and $\gamma \in (-3,0]$. Suppose that a curve $\mu: [0,T] \to \mathcal{P}(\R^3)$ has a density $f(t, v)$ that satisfies the following assumptions:
    \begin{enumerate}[1.]
        \item For $\kappa \in (0, \gamma+3]$, we have 
        $\langle v\rangle^{2-\gamma} f(t, v) \in L_t^{\infty}(0,T; L_v^1 \cap L_v^{(3-\kappa)/(3+\gamma-\kappa)}(\R^3))$, where $\langle v \rangle^2 = 1+|v|^2$;
        \item The initial entropy is finite: $H(f(0, \cdot)) = \int_{\R^3} f(0, v) \log f(0, v) \rd v < \infty$;
        \item The entropy-dissipation is time integrable, i.e., $f$ satisfies \eqref{entropy_diss}.
    \end{enumerate}
    Then $\mu$ is a curve of maximal slope for $H$ with respect to its strong upper gradient $\sqrt{D_{H}}$ if and only if its density $f$ is a weak solution of the homogeneous Landau equation \eqref{landau}. Moreover, we have the following energy dissipation equality:
    \begin{equation*}
        H(f(t, \cdot)) + \frac12 \int_0^t |f'|^2_{d_L}(s) \rd s + \frac12 \int_0^t D_{H}(f(s, \cdot)) ds = H(f(0, \cdot)) \,,
    \end{equation*}
    where $|f'|_{d_L}$ is the metric derivative of curve $f$ with metric $d_L$  defined in \eqref{dL}. 
\end{theorem}
For foundational theory on gradient flows, one can refer to \cite[Chapter 1]{Ambrosio2005GradientFI} and consult Appendix~\ref{app_def} for basic definitions. The key distinction here lies in the definition of the metric, which has been shown in \cite[Theorem 7]{carrillo2024landau} to provide a meaningful topology on $\mathcal{P}_{2, E}$. As mentioned earlier, $d_L$ in its original form is not easy to compute. Therefore, we derive an alternative form of the Landau metric. 

Recall the weak form \eqref{landau2} of the Landau equation. If we choose $\varphi = \log f(v) = \frac{\delta H}{\delta f}$ with $H$ given in \eqref{0910}, we get entropy decay with the decay rate $-\frac{1}{2}\iint_{\R^{2d}} (S - S_*) \cdot A(v - v_*)(S - S_*) ff_* \rd v\rd v_*$, where $S(v):= \nabla_v \log f(v)$. This motivates the definition of an action functional 
\begin{equation} \label{0911}
    \frac{1}{2}\iint_{\R^{2d}} (s-s_*) \cdot A(v-v_*)(s-s_*) f f_* \rd v\rd v_* \,, ~ \text{where}~~ s = s(v), ~s_* = s(v_*)\,.
\end{equation}
Following the minimizing action principle, one can minimize \eqref{0911} over the curves that satisfy an appropriate continuity equation to derive the Landau metric. To construct such a continuity equation, we interpret the Landau equation as a continuity equation with an integral velocity field,  i.e., 
\begin{equation} \label{0912}
    \partial_t f = \nabla \cdot \left[f \left( \int_{\R^{d}} A(v-v_*) (s-s_*)f_{*} \rd v_* \right)\right]\,.
\end{equation}
The combination of \eqref{0911} and \eqref{0912} then gives rise to the Landau metric \eqref{dL-new}. 

A more direct relation between \eqref{dL} and \eqref{dL-new} is to observe that, by comparing the constrained equation in \eqref{dL} with the Landau equation \eqref{landau3}, we see that $V$ corresponds to $\tilde \nabla \frac{\delta H}{\delta f}$. Therefore, we can assume $V$ admits the following form: 
\begin{equation*}
    V = \tilde{\nabla}\phi = |v-v_*|^{1+\gamma/2} \Pi(v-v_*)(\nabla\phi - \nabla_*\phi_*) \,,
\end{equation*}
for some function $\phi=\phi(t, v)$. This is analogous to the fact that the optimal vector field in the 2-Wasserstein metric is the gradient of some potential function \cite{villani2003topics}. Denote $s := \nabla\phi$ and $|s-s_*|^2_{A(v-v_*)} := (s-s_*) \cdot A(v-v_*) (s-s_*)$. Then \eqref{dL-new} directly follows from \eqref{dL}.

%%%%%%%%%%%%%%%%%%%%%%%%%%%
\subsection{Dynamic JKO scheme}
With the definition of the Landau metric established, we can now formulate the dynamic JKO scheme. This approach was first explored as a viable numerical method for Wasserstein gradient flow in \cite{Carrillo2022primaldual} and subsequently in \cite{Carrillo2024primaldual} for more general metrics. 
\begin{problem}[dynamic JKO scheme]
Given $f^n$, solve $f^{n+1} := f(1, \cdot)$ by
\begin{equation}\label{djko}
    \left \{
    \begin{aligned}
        & \inf_{f,s} ~ \int_0^1 \frac12 \iint_{\R^{2d}} |s-s_*|^2_{A(v-v_*)} f f_{*} \rd v \rd v_* \rd t + 2\Delta t C_{\gamma} H(f(1, \cdot)) \,, \\
        & s.t.~~
        \partial_t f = \nabla \cdot \left[f \left(\int_{\R^d} A(v-v_*) (s-s_*) f_{*} \rd v_* \right)\right] \,,\ f(0, \cdot) = f^n \,.
    \end{aligned}
    \right.
\end{equation}
\end{problem}
As with the vanilla JKO scheme, \eqref{djko} enjoys a structure-preserving property.
\begin{proposition} \label{prop2}
    Formulation \eqref{djko} has the following properties for any $n \geq 0$:
    \begin{enumerate}[(1.)]
        \item Entropy dissipation: $H(f^{n+1}) \leq H(f^n)$.
        \item Mass, momentum, and energy conservation: 
        \[
        \int_{\R^d} \varphi(v) f^{n+1}(v) \rd v = \int_{\R^d} \varphi(v) f^n(v) \rd v, ~\text{for}~  \varphi(v) = 1\,, v\,, |v|^2\,.
        \]  
    \end{enumerate}
\end{proposition}
\begin{proof}
    Property 1 is a direct consequence of the fact that $f^{n+1}$ is a minimizer. 
    Property 2 is guaranteed by the constraint PDE: ~for $\varphi(v) = 1\,, v\,, |v|^2$,
    \begin{flalign*}
        \frac{\rd}{\rd t} \int_{\R^d} f \varphi \rd v 
        & = \int_{\R^d} \varphi \nabla \cdot \left[f \left(\int_{\R^d} A(v-v_*) (s-s_*) f_{*} \rd v_* \right)\right] \rd v \\
        & = -\iint_{\R^{2d}} \nabla\varphi  \cdot A(v-v_*) (s-s_*) ff_{*} \rd v\rd v_* \\
        & = -\frac{1}{2} \iint_{\R^{2d}} (\nabla\varphi - \nabla_*\varphi_*) \cdot A(v-v_*) (s-s_*) f f_{*} \rd v \rd v_* =0\,.
    \end{flalign*}
\end{proof}

We now demonstrate that \eqref{djko} indeed provides a first-order approximation in $\Delta t$ to \eqref{landau}. Let $\lambda(t, v)$ be the Lagrangian multiplier, the Lagrangian associated with  \eqref{djko} is given by:
\begin{flalign*}
    \mathcal{L}(f, s, \lambda) 
    = & ~\int_0^1 \frac12 \iint_{\R^{2d}} (s-s_*) \cdot A(v-v_*)(s-s_*)ff_* \rd v \rd v_* \rd t + 2 \Delta t C_\gamma H(f(1, \cdot)) \\
    & + \int_0^1 \int_{\R^d} \lambda \left[ \partial_t f - \nabla \cdot \left(f \int_{\R^d} A(v-v_*) (s-s_*)f_* \rd v_* \right)\right] \rd v \rd t \\
    = & ~\int_0^1 \frac12 \iint_{\R^{2d}} (s-s_* + \nabla\lambda - \nabla_*\lambda_*) \cdot A(v-v_*)(s-s_*)ff_* \rd v \rd v_* \rd t \\
    & - \int_0^1 \int_{\R^d} f \partial_t \lambda \rd v \rd t + \int_{\R^d} f \lambda |_{t=0}^{t=1} \rd v + 2 \Delta t C_\gamma H(f(1, \cdot)) \,.
\end{flalign*}
By definition, the variation of $\mathcal{L}$ with respect to $s$ is calculated via
\begin{flalign*}
    \int_0^1 \int_{\R^d} \eta \frac{\delta \mathcal{L}}{\delta s} \rd v \rd t
    & = \lim_{\varepsilon \to 0} \frac{\mathcal{L}(f, s + \varepsilon \eta, \lambda) - \mathcal{L}(f, s, \lambda)}{\varepsilon} \\
    & = \int_0^1 \frac12 \iint_{\R^{2d}} (2s-2s_* + \nabla\lambda - \nabla_*\lambda_*) \cdot A(v-v_*)(\eta-\eta_*)ff_* \rd v \rd v_* \rd t \\
    & = \int_0^1 \int_{\R^d} \eta \left[f \left( \int_{\R^d} A(v-v_*) (2s-2s_* + \nabla\lambda - \nabla_*\lambda_*) f_* \rd v_* \right)\right] \rd v \rd t \,.
\end{flalign*}
Therefore, the first-order optimality condition is written as:
\begin{equation}\label{opt_u}
    \frac{\delta \mathcal{L}}{\delta s} = f \left(\int_{\R^d} A(v-v_*) (2s-2s_* + \nabla\lambda - \nabla_*\lambda_*)f_* \rd v_* \right) = 0 \,.
\end{equation}
Likewise, we have the other two optimality conditions: 
\begin{flalign}
    \frac{\delta \mathcal{L}}{\delta f} &= -\partial_t \lambda + \int_{\R^d} (s-s_* + \nabla\lambda - \nabla_*\lambda_*) \cdot A(v-v_*)(s-s_*) f_* \rd v_* = 0 \,, \label{opt_rho} \\
    \frac{\delta \mathcal{L}}{\delta f(1, \cdot)} & = \lambda(1, \cdot) + 2\Delta t C_\gamma \frac{\delta H(f(1, \cdot))}{\delta f(1, \cdot)} = \lambda(1, \cdot) + 2\Delta t C_\gamma \log f(1, \cdot) = 0 \label{opt_rho1} \,.
\end{flalign}

\begin{lemma}\label{lemma_1}
    If $\varphi(v) : \R^d \to \R^d$ satisfies $\int_{\R^d} A(v-v_*) (\varphi -\varphi_*)f_* \rd v_* = 0$, then $\varphi \in \operatorname{span}\{ \mathbf{1}, v\}$ on the support of $f$.
\end{lemma}
\begin{proof}
    Integrating this equality against $f \rd v$, we have 
    \begin{equation*}
        0 \!=\! \iint_{\R^{2d}} \varphi \cdot A(v - v_*) (\varphi -\varphi_*)f_*f \rd v_* \rd v
        \!=\! \frac{1}{2} \iint_{\R^{2d}} (\varphi -\varphi_*) \cdot A(v-v_*) (\varphi -\varphi_*)f_*f \rd v_* \rd v \,.
    \end{equation*}
    Since $A$ is semi-positive definite, this implies $A(v-v_*)(\varphi -\varphi_*) = 0$ on the support of $f_* f$. Therefore, $\varphi \in \operatorname{span}\{ \mathbf{1}, v\}$ on the support of $f$.
\end{proof}

By Lemma \ref{lemma_1}, \eqref{opt_u} implies that $2s + \nabla\lambda \in \operatorname{span}\{ \mathbf{1}, v\}$ on $\operatorname{supp}(f)$. Combined with \eqref{opt_rho1}, this shows that 
\begin{equation}\label{opt_u1}
    s(1, v) \in \Delta t C_\gamma \nabla\log f(1, v) + \operatorname{span}\{ \mathbf{1}, v\} \, ~~\text{on}~~ \operatorname{supp}(f)\,.
\end{equation}
Substituting 
\eqref{opt_u1} into the constraint PDE in \eqref{djko} reveals that $f^{n+1} = f(1, \cdot)$ is indeed the solution of the homogeneous Landau equation \eqref{landau} after one time step $\Delta t$.
Additionally, \eqref{opt_rho} shows that $\lambda$ satisfies the following equation
\begin{equation*}
    \partial_t \lambda + \frac{1}{4} \int_{\R^d} |\nabla\lambda - \nabla_*\lambda_*|^2_{A(v-v_*)} f_* \rd v_* = 0 \,.
\end{equation*}
This can be viewed as a generalization of the Hamilton-Jacobi equation that the dual variable must satisfy in the 2-Wasserstein metric.

%%%%%%%%%%%%%%%%%%%%%%%%%%%%%%%%%
\subsection{Lagrangian formulation}
To facilitate the particle method in the next section, we translate the variational problem \eqref{djko} into the corresponding Lagrangian formulation, following \cite{CARRILLO2021271, LEE2024113187}. Assuming the velocity field is sufficiently regular, the solution $f(t, v)$ to the constrained continuity equation can be represented as
\begin{equation*}
    f(t, \cdot) = \T_t \# f^n \,,\ t \in [0,1] \,,
\end{equation*}
where $\T_t$ is the flow map that solves the following ODE:
\begin{equation}\label{flowmap}
    \frac{\rd}{\rd t} \T_t(v) = - \int_{\R^d} A(\T_t(v)-\T_t(v_*)) [s(t, \T_t(v)) - s(t, \T_t(v_*))] f^n_* \rd v_* \,, 
\end{equation}
with $\T_0(v)=v $.
The entropy term in \eqref{djko} can also be represented using map $\T_t$: 
\begin{flalign*}
    H({\T_1}\# f^n) 
    & = H(f(1, \cdot))= \int_{\R^d} f(1, v) \log f(1, v) \rd v \\
    & = \int_{\R^d} f(1, \T_1(v)) \log f(1, \T_1(v))|\det\nabla_v \T_1(v)| \rd v \\
    & = \int_{\R^d} f^n(v)\log\frac{f^n(v)}{|\det\nabla_v \T_1(v)|} \rd v \\
    & = \int_{\R^d} f^n(v)\log{f^n(v)} \rd v - \int_{\R^d} f^n(v)\log|\det\nabla_v \T_1(v)| \rd v \,,
\end{flalign*}
where we have used the change of variable formula $  f(1, \T_1(v)) |\det\nabla_v \T_1(v)| = f^n(v)$. In practice, $\int_{\R^d} f^n(v)\log{f^n(v)} \rd v$ is dropped in optimization since it is independent of $s$. To efficiently compute $|\det\nabla_v \T_1(v)|$, we cite a formula in \cite{HUANG2025114053} on the evolution of the log determinant, inspired by the concept of continuous normalizing flow \cite{neuralode, grathwohl2018scalable, LEE2024113187}.

\begin{proposition}\cite[Proposition 4]{HUANG2025114053}\label{prop:logdet}
Assuming the flow map $\T_t$ in \eqref{flowmap} is invertible, then the log determinant of its Jacobian satisfies the following equation:
\begin{flalign*}
    \frac{\rd}{\rd t} \log |\det\nabla_{v} \T_t(v)|
    & = - \int_{\R^d} \nabla_{\T_t(v)} \cdot \{A(\T_t(v)-\T_t(v_*)) [s(t, \T_t(v)) - s(t, \T_t(v_*))] \} f^n_* \rd v_* \\
    & = - \int_{\R^d} \big\{ A(\T_t(v)-\T_t(v_*)) : \nabla_{\T_t(v)} s(t, \T_t(v)) - (d-1) \\
    & |\T_t(v)-\T_t(v_*)|^{\gamma} (\T_t(v)-\T_t(v_*)) \cdot [s(t, \T_t(v)) - s(t, \T_t(v_*))] \big\} f^n_* \rd v_* \,,
\end{flalign*}
where $A:B := \sum_{ij}A_{ij}B_{ij}$.
\end{proposition}

As a result, we propose the following Lagrangian dynamic JKO scheme. 
\begin{problem}[dynamic JKO scheme in Lagrangian formulation]
Given $f^n$, find $f^{n+1} := \T_1^{n+1} \# f^n$, where $\T_t^{n+1}$ is the optimal flow map solving
\begin{equation}\label{djko2}
    \left \{
    \begin{aligned}
        & \inf_{s} \int_0^1 \frac12 \! \iint_{\R^{2d}} |s(t, \T_t(v)) \!-\! s(t, \T_t(v_*))|^2_{A(\T_t(v) - \T_t(v_*))} f^n f^n_* \rd v \rd v_* \rd t \!-\!
        2\Delta t C_{\gamma} \! \int_{\R^d} \log |\det\nabla_v \T_1(v)| f^n \rd v , \\
        & s.t.~~ \frac{\rd}{\rd t} \T_t(v) = - \int_{\R^d} A(\T_t(v) - \T_t(v_*)) [s(t, \T_t(v)) - s(t, \T_t(v_*))] f^n_* \rd v_* \,, \\
        & \frac{\rd}{\rd t} \log |\det\nabla_v \T_t(v)| 
        = - \int_{\R^d} \nabla_{\T_t(v)} \cdot \{A(\T_t(v) - \T_t(v_*))[s(t, \T_t(v)) - s(t, \T_t(v_*))] \} f^n_* \rd v_* \,, \\
        & \T_0(v)=v \,,~ \log|\det\nabla_v \T_0(v)|=0 \,.
    \end{aligned}
    \right.
\end{equation}
\end{problem}
The formulation \eqref{djko2} is equivalent to \eqref{djko} and therefore preserves all the properties stated in Proposition~\ref{prop2}.

%%%%%%%%%%%%%%%%%%%%%%%%%%%%%%%%%%%%%%%%%
\section{A particle method}\label{sec3}
In this section, we derive a fully implementable version of \eqref{djko2} by first discretizing $v$ using particles and approximating the to-be-optimized vector field $s(t, v)$ with the neural network.

%%%%%%%%%%%%%%%%%%%%%%%%%%%%
\subsection{Particle representation}
By interpreting the integrals against $f^n\rd v$ and $f^n_*\rd v_*$ as expectations, \eqref{djko2} reveals a particle representation. More precisely, let $\{ v_i^n \}_{i=1}^{N}$ be the velocity of $N$ particles sampled from $f^n$, we discretize \eqref{djko2} as follows.

\begin{problem}[dynamic JKO scheme using particles]
Given particles $\{v_i^n\}_{i=1}^N$, find $v_i^{n+1} := \T_1^{n+1}(v_i^n)$ and $f^{n+1}(v_i^{n+1}):=f^n(v_i^n) / |\det \nabla_v \T_1^{n+1}(v_i^n)|$, where $\T_t^{n+1}$ is the optimal flow map  solving
\begin{equation}\label{pjko}
\left \{
    \begin{aligned}
        & \inf_{s} ~ \int_0^1 \frac{1}{2N^2} \sum_{i=1}^{N}\sum_{j=1}^{N} |s(t, \T_t(v_i^n)) - s(t, \T_t(v_j^n))|^2_{A(\T_t(v_i^n) - \T_t(v_j^n))} \rd t - 2\Delta t C_{\gamma} \frac{1}{N} \sum_{i=1}^{N} \log |\det\nabla_v \T_1(v_i^n)| \,, \\
        & \text{s.t.}~~ \frac{\rd}{\rd t} \T_t(v_i^n) = -\frac1N \sum_{j=1}^N A(\T_t(v_i^n) - \T_t(v_j^n)) [s(t, \T_t(v_i^n)) - s(t, \T_t(v_j^n))] \,, \\
        & \frac{\rd}{\rd t} \log |\det\nabla_v \T_t(v_i^n)| = - \frac1N \sum_{j=1}^{N} \nabla_{\T_t(v_i^n)} \cdot \{A(\T_t(v_i^n) - \T_t(v_j^n))[s(t, \T_t(v_i^n)) - s(t, \T_t(v_j^n))] \} \,, \\
        & \T_0(v_i^n)=v_i^n \,,~ \log |\det \nabla_v \T_0(v_i^n)|=0 \,. 
    \end{aligned}
    \right.
\end{equation}
\end{problem}

The particle formulation \eqref{pjko} immediately has the following favorable properties.
\begin{proposition}\label{prop_31}
    The particle-based variational formulation \eqref{pjko} has the following properties for any $n \geq 0$:
    \begin{enumerate}[(1.)]
        \item Discrete entropy dissipation: define $H^N(f^{n}) := \frac{1}{N} \sum_{i=1}^N \log f^n(v_i^{n})$, then $H^N(f^{n+1}) \leq H^N(f^{n})$.
        \item Discrete mass, momentum, and energy conservation: 
        \[
        \frac{1}{N} \sum_{i=1}^N \varphi(v_i^{n+1}) = \frac{1}{N} \sum_{i=1}^N \varphi(v_i^n) \quad\text{for}\quad  \varphi(v) = 1\,, v\,, |v|^2 \,.
        \]
    \end{enumerate}
\end{proposition} 
\begin{proof}
    Property 1 is a direct consequence of optimization. Let $s^{n+1}$ be the minimizer of the discrete variational problem \eqref{pjko} and denote by 
    $$d_L^{2, N}(s) \!:=\! \int_0^1 \frac{1}{2N^2} \sum_{i=1}^{N}\sum_{j=1}^{N}|s(t, \T_t(v_i^n)) - s(t, \T_t(v_j^n))|^2_{A(\T_t(v_i^n) - \T_t(v_j^n))} \rd t \,.$$ 
    Then we have
    \begin{flalign*}
        2\Delta t H^N(f^{n+1}) + \tfrac{1}{C_\gamma} d_L^{2, N}(s^{n+1})
        & = 2\Delta t H^N(f^{n}) - \tfrac{2\Delta t}{N}\sum_{i=1}^{N} \log |\det\nabla_v \T_1^{n+1}(v_i^n)| + \tfrac{1}{C_\gamma} d_L^{2, N}(s^{n+1}) \\
        & \leq 2\Delta t H^N(f^{n}) \,,
    \end{flalign*}
    where the last inequality holds by simply taking $s=\boldsymbol{0}$ in \eqref{pjko}. Hence, by the non-negativity of $d_L^{2, N}$, the result holds.
    
    Property 2 is true thanks to the flow map constraint: for $\varphi(v) = 1\,, v\,, |v|^2$, then
    \begin{flalign*}
        & \frac{\rd}{\rd t} \frac{1}{N} \sum_{i=1}^N \varphi(\T_t(v_i^n)) = \frac{1}{N} \sum_{i=1}^N \nabla\varphi(\T_t(v_i^n)) \cdot \frac{\rd}{\rd t} \T_t(v_i^n) \\
        = & -\frac{1}{N^2} \sum_{i=1}^{N}\sum_{j=1}^{N} \nabla\varphi(\T_t(v_i^n)) \cdot A(\T_t(v_i^n) - \T_t(v_j^n)) [s(t, \T_t(v_i^n)) - s(t, \T_t(v_j^n))] \\
        = & -\frac{1}{2N^2} \sum_{i=1}^{N}\sum_{j=1}^{N} [\nabla\varphi(\T_t(v_i^n)) - \nabla\varphi(\T_t(v_j^n))] \cdot A(\T_t(v_i^n) - \T_t(v_j^n)) [s(t, \T_t(v_i^n)) - s(t, \T_t(v_j^n))] = 0 \,.
    \end{flalign*}
\end{proof}

\subsection{Neural network approximation}
To implement \eqref{pjko}, we need to represent the vector field $s(t, v)$ in a finite-dimensional space. Here, we use a neural network and denote the approximation by $s_\theta(t,v)$, where $\theta$ represents the trainable network parameters. Compared with traditional representations such as polynomial or Fourier bases, neural networks scale more favorably with dimension and can naturally incorporate time dependence. Nevertheless, neural networks are neither the only nor a strictly necessary choice for approximation; they are simply a viable option. Given the modest size of the network employed in our implementation, training does not pose a significant challenge.

Specifically, the inner time interval $[0,1]$ is first discretized into $N_{\tau}$ uniform subintervals $[\tau_k, \tau_{k+1}]$ with inner time step $\tau = \frac{1}{N_{\tau}}$, where $k=0,1,\ldots, N_{\tau}-1$. For simplicity in illustrating the fully discretized JKO scheme, we solve the ODE constraints in \eqref{pjko} by the forward Euler method. Alternative high-order ODE solvers, such as Runge–Kutta methods, could also be employed in this framework. Given the samples $\{v_i^n\}_{i=1}^N$ at time $t^n$, the loss function at the $n+1$-th JKO step is defined as:
\begin{equation}\label{jko_loss}
    \inf_{\theta} ~ \frac{1}{N^2} \sum_{i=1}^{N}\sum_{j=1}^{N} \left[ \tau \sum_{k=0}^{N_{\tau}-1} \frac{1}{2}|s_\theta(\tau_k, z_i^{k}) - s_\theta(\tau_k, z_j^{k})|^2_{A(z_i^{k} - z_j^{k})} \right] - 2\Delta t C_\gamma \frac{1}{N} \sum_{i=1}^{N} h_{i}^{N_{\tau}} \,,
\end{equation}
where the inner update equations are given by
\begin{equation}\label{update_1}
\begin{split}
    & z_i^{k+1} = z_i^{k} - \tau \frac{1}{N} \sum_{j=1}^N A(z_i^{k} - z_j^{k}) [s_\theta(\tau_k, z_i^{k}) - s_\theta(\tau_k, z_j^{k})] \,, \\
    & h_{i}^{k+1} = h_{i}^{k} - \tau \frac{1}{N} \sum_{j=1}^N \big\{ A(z_i^{k} - z_j^{k}) : \nabla s_\theta(\tau_k, z_i^{k}) - (d - 1) |z_i^{k} - z_j^{k}|^{\gamma}(z_i^{k} - z_j^{k}) \cdot [s_\theta(\tau_k, z_i^{k}) - s_\theta(\tau_k, z_j^{k})] \big\} \,, \\
    & z_i^0 \!=\! v_i^n \,,~~ h_i^0 = 0 \,.
\end{split}
\end{equation}
Here we use Proposition \ref{prop:logdet} to compute $h_{i}^{k+1}$. The loss function \eqref{jko_loss} is then minimized using off-the-shelf optimizers. Once the optimal $s^{n+1}_{\theta}$ is learned, we update the particle velocities and densities at $t^{n+1}$ via
\begin{equation}\label{update_2}
    v_i^{n+1} := z_i^{N_{\tau}} \,,~~ f^{n+1}(v_i^{n+1}) := \frac{f^n(v_i^n)}{\exp{\left(h_{i}^{N_{\tau}}\right)}} \,.
\end{equation}
From now on, we refer to procedure \eqref{jko_loss}--\eqref{update_2} as the \emph{JKO--$N_{\tau}$ particle method}.

\begin{proposition} \label{prop32}
   The JKO--$N_{\tau}$ particle method \eqref{jko_loss}--\eqref{update_2} has the following properties for any $n \geq 0$:
    \begin{enumerate}[1.]
        \item Discrete entropy dissipation: $H^N(f^{n+1}) \leq H^N(f^{n})$.
        \item Discrete mass and momentum conservation: $\frac{1}{N} \sum_{i=1}^N \varphi(v_i^{n+1}) = \frac{1}{N} \sum_{i=1}^N \varphi(v_i^n)$ for $\varphi(v) = 1\,, v$.
        \item Discrete energy $\frac{1}{N} \sum_{i=1}^N |v_i^n|^2$ is conserved up to $O(\Delta t)$ when using a one-step forward Euler discretization in the inner time, i.e., $N_{\tau} = 1$.
    \end{enumerate}
\end{proposition}
\begin{proof}
    Property 1 follows the proof of property 1 in Proposition \ref{prop_31}. Properties 2 and 3 are similar to the proof of \cite[Proposition 2]{HUANG2025114053}, so we omit them here.
\end{proof}

\begin{remark} \label{remark1}
   We note that the above proposition assumes exact minimization, i.e., that the infimum in \eqref{jko_loss} is attained. In practice, however, this is difficult to guarantee. Nevertheless, part 1 of the proposition requires only that the loss function in \eqref{jko_loss} be non-positive, which is straightforward to ensure numerically. Parts 2 and 3 hold regardless of whether the infimum is attained, as they follow directly from the structural properties of the matrix $A$ in \eqref{update_1}, independent of the specific value of $s_\theta$ used.
\end{remark}

The procedure of the JKO--$N_{\tau}$ particle method is summarized in Algorithm \ref{algo:jko}.

\begin{algorithm}[tp]
\caption{JKO--$N_{\tau}$ particle method for Landau equation}
\label{algo:jko}
\begin{algorithmic}[1]
    \Require{$N$ initial particles $\{v_i^0\}_{i=1}^N \stackrel{i.i.d.}\sim f^0$; outer JKO time step $\Delta t$; inner time step $\tau$ and total number of inner steps $N_{\tau}$; total number of JKO steps $N_T$.}
    \Ensure{Neural networks $s^n_\theta$, velocity of particles $\{v_i^n\}_{i=1}^N$ and densities $\{f^{n}(v_i^{n})\}_{i=1}^N$ for all $n=1, \cdots, N_T$.}

    \For{$n=0,\cdots,N_T-1$}
        \State Initialize neural network $s_\theta$ for $t^{n+1}$.
        \While{not converged} 
            \State Update parameter $\theta$ of $s_\theta$ by minimizing the JKO loss \eqref{jko_loss} using Algorithm \ref{algo:opt}.
        \EndWhile
        \For{$i=1,\cdots,N$}
            \State Obtain $v_i^{n+1}$ from $v_i^{n}$ and $f^{n+1}(v_i^{n+1})$ from $f^n(v_i^n)$ via \eqref{update_1}--\eqref{update_2} using the optimal $s^{n+1}_\theta$.
        \EndFor
    \EndFor
\end{algorithmic}
\end{algorithm}

\begin{remark}
    The exact energy conservation could be achieved using the implicit midpoint discretization or, more generally, a discrete gradient integrator in time \cite{Hirvijoki_2021}. The trade-off is that this method requires an implicit update for $v^{n+1}_i$, which necessitates a fixed-point iteration. We leave the exploration of this method for future work.
\end{remark}
\begin{remark}
    Unlike several recent works that use neural networks to learn or solve the Landau collision operator \cite{Miller_Churchill_Dener_Chang_Munson_Hager_2021, LEE2023112031, CHUNG2023112400, NOH2025113665}, our approach directly solves the Landau equation and is entirely training-data-free. Moreover, our approach is computationally efficient in training, as shown in Fig. \ref{fig_eg4_time} in a later section. Furthermore, new trajectories from the same initial data can be generated by first storing the trained neural networks and then applying \eqref{update_1}--\eqref{update_2}. 
\end{remark}

%%%%%%%%%%%%%%%%%%%%%%%%%%%%%%%%%%%%%%%%%
\section{Stochasticity-accelerated JKO scheme}\label{sec4}
One advantage of \eqref{jko_loss} is that it is well-suited for stochastic methods, which are crucial for high-dimensional problems due to their efficiency and reduced memory consumption. In this section, we provide a detailed discussion on leveraging stochasticity in optimization and examine its convergence. Once $s^{n+1}_\theta$ is learned, particle updates can also be accelerated using the random batch method \cite{Carrillo_Jin_Tang_2022}, with further details provided in Section~\ref{sec:rbm}.

%%%%%%%%%%%%%%%%%%%%%%%%%%%
\subsection{Stochastic optimization}
Comparing the proposed JKO-based particle method with the score-based particle method \cite{HUANG2025114053, ilin2024}, the primary difference lies in their loss functions. The JKO loss function is designed to respect entropy dissipation, and the resulting scheme therefore enjoys exact entropy decay and unconditional stability. However, this comes at the cost of a more complex loss function. The score-based approach employs the following loss function:
\begin{equation*}
    \frac{1}{N} \sum_{i=1}^N |s^n_\theta(v_i^n)|^2 + 2 \nabla \cdot s^n_\theta(v_i^n) \,,
\end{equation*}
which requires only $O(N)$ computational complexity. In contrast, the JKO method \eqref{jko_loss}--\eqref{update_1} involves a double summation, resulting in $O(N^2)$ complexity. Furthermore, the practical implementation of \eqref{jko_loss}--\eqref{update_1} often leads to issues such as GPU out-of-memory errors when $N$ becomes large.

Nevertheless, a common approach to address these issues is to apply stochastic optimization algorithms. If we incorporate the update equation for $h_i^{k}$ in \eqref{update_1} directly into the loss \eqref{jko_loss}, the loss function can be written as
\begin{equation*}
    \ell(\theta) = \frac{1}{N^2} \sum_{i=1}^{N}\sum_{j=1}^{N} \ell_{i,j}(\theta) \,,
\end{equation*}
for some suitable $\ell_{i,j}(\theta)$. In standard mini-batch SGD \cite{Bottou2018}, the indices $i$ and $j$ are treated as independent, and a batch of index pairs $(i, j) \in [N] \times [N]$ is randomly selected to compute the stochastic gradient. However, in our case, $i$ and $j$ represent two interacting particles. Therefore, it is crucial to preserve their relationship. To address this, we first randomly select a batch of indices from $[N]$ and then update parameter $\theta$ based on the gradient information within this batch. 

To further speed up the training process, we adopt the shuffling-type gradient method \cite{sgd_shuffling}. That is, for each epoch, we choose a batch size $B$ ($B \ll N$) and randomly divide $N$ indices into $\frac{N}{B}$ batches, denoted by $C_q$, $q=1,\ldots, \frac{N}{B}$. We then perform one step of gradient descent within each $C_q$ successively. Thus, the cost of one gradient descent decreases to $O(BN)$. The random division can be realized in $O(N)$ through random permutation via a PyTorch function such as `torch.randperm'. Consequently, the total computational cost per epoch is significantly reduced to $O(BN)$. In practice, the batch size is chosen heuristically to balance accuracy and computational time, and the vanilla gradient descent is substituted with more efficient optimizers. The procedure is summarized in Algorithm \ref{algo:opt}.

\begin{algorithm}[tp]
\caption{mini-batch SGD with random reshuffling}
\label{algo:opt}
    \begin{algorithmic}[1]
        \Require Initial parameters $\theta_0$; mini-batch size $B$; epoch numbers $K$.
        \Ensure Trained parameters $\theta_K$.
        \For{epoch $k=1,\cdots, K$}
            \State Divide indices $[N]$ into $\frac{N}{B}$ batches randomly.
            \State Set $\theta^{(k)}_0=\theta_{k-1}$.
            \For{each batch $C_q$}
                \State $\theta^{(k)}_q = \theta^{(k)}_{q-1} - \alpha^{(k)}_q \frac{1}{B^2} \sum_{i \in C_q}\sum_{j \in C_q} \nabla_{\theta} \ell_{i,j}(\theta^{(k)}_{q-1})$, where $\alpha^{(k)}_q$ is the learning rate.
            \EndFor
            \State Set $\theta_k=\theta^{(k)}_{\frac{N}{B}}$.
        \EndFor
    \end{algorithmic} 
\end{algorithm}

\begin{remark}
    It is worth noting that the randomness introduced in the method does not violate the properties summarized in Proposition~\ref{prop32}. This is because, within each batch, the constraint in \eqref{pjko} naturally conserves mass, momentum, and energy. Additionally, the entropy decay property, which, as noted in Remark~\ref{remark1}, only requires verifying that the objective function is non-positive. In practice, we monitor the loss value for each batch and ensure that the sum over all batches is non-positive.
\end{remark}

%%%%%%%%%%%%%%%%%%%%%%%%%%%
\subsection{Convergence analysis for optimization}
Following the analysis of the SGD with random reshuffling \cite{sgd_shuffling}, in this section, we show the convergence of Algorithm \ref{algo:opt}. 
We first make the following assumptions, which are standard in optimization theory.
\begin{unlist}
    \item[(A.1)]~~ $dom(\ell) := \{ \theta \in \R^{d_\theta} : \ell(\theta) < +\infty \} \not= \emptyset$ and $\ell_* := \inf_{\theta} \ell(\theta) > -\infty$. 
    \item[(A.2)]~~ $\ell_{i,j}(\theta)$ is $L$-smooth, i.e.
    $\| \nabla_\theta \ell_{i,j}(\theta_2) - \nabla_\theta \ell_{i,j}(\theta_1) \| \leq L \| \theta_2 - \theta_1 \|$,  for all $i,j \in [N]$. 
    \item[(A.3)]~~ $\nabla \ell_{i,j}$ has bounded variance: $\exists$ two constants $M \geq 0$ and $\sigma >0$ such that for all $\theta \in \R^{d_\theta}$, we have $\frac{1}{N^2} \sum_{i=1}^{N}\sum_{j=1}^{N} \left\| \nabla\ell_{i,j}(\theta) - \nabla\ell(\theta) \right\|^2 \leq M \left\| \nabla\ell(\theta) \right\|^2 + \sigma^2 \,.$
\end{unlist}

We note that assumption (A.2) can be relaxed in the analysis of shuffling-type gradient descent methods. In \cite{pmlr-v267-he25u}, it is only required that $\ell$ and $\ell_{i,j}$ be $g$-smooth for some sub-quadratic function $g$. Here, by $g$-smooth we mean that $\| \nabla_\theta^2 \ell(\theta) \| \leq g(\| \nabla_\theta \ell(\theta) \|)$ almost everywhere for some non-decreasing continuous function $g$.

The main result is as follows. 
\begin{theorem}\label{thm_opt_conv}
    Suppose the above assumptions hold. Let learning rate $\alpha^{(k)}_q := \alpha \frac{B}{N} > 0$ for $0 < \alpha \leq \frac{1}{L\sqrt{2(3M+2)}} \frac{B}{N}$. Then after $K$ epochs we have,
    \begin{equation*}
        \frac{1}{K} \sum_{k=0}^{K-1} \mathbb{E} \left[ \left\| \nabla_{\theta} \ell(\theta_k) \right\|^2 \right]
        \leq \frac{4}{\alpha K} \left[\ell(\theta_0) - \ell_*\right] + 6\sigma^2 L^2 \alpha^2 \frac{N}{B^2} \,.
    \end{equation*}
\end{theorem}
As mentioned earlier, the main difference between our method and the vanilla mini-batch SGD is the treatment of the indices $i$ and $j$. Consequently, the corresponding gradient needs to be estimated differently. 
\begin{proof}
Combining Lemma \ref{sgd_lemma1} and \ref{sgd_lemma2}, we have
\begin{flalign*}
    \ell(\theta_{k+1}) 
    & \leq \ell(\theta_{k}) 
    + \frac{\alpha L^2 B}{2N} \sum_{q=1}^{\frac{N}{B}} \left\| \theta_{k} - \theta^{(k+1)}_{q-1} \right\|^2
    - \frac{\alpha}{2} \left\| \nabla\ell(\theta_{k}) \right\|^2 \\
    & \leq \ell(\theta_{k}) 
    + \frac{\alpha L^2}{2N} \frac{\alpha^2 N^2}{B^2} \left( (3M+2) \left\| \nabla_{\theta} \ell(\theta_k) \right\|^2 + 3\sigma^2 \right) 
    - \frac{\alpha}{2} \left\| \nabla\ell(\theta_{k}) \right\|^2 \\
    & = \ell(\theta_{k}) 
    + \frac{\alpha}{2} \left( \frac{\alpha^2 L^2 N}{B^2}(3M+2) - 1 \right) \left\| \nabla_{\theta} \ell(\theta_k) \right\|^2 + \frac{3\alpha_k^3 \sigma^2 L^2 N}{2B^2} \,. 
\end{flalign*}
Since $0 < \alpha \leq \frac{1}{L\sqrt{2(3M+2)}} \frac{B}{N}$, we have $\frac{\alpha^2 L^2 N}{B^2}(3M+2) - 1 \leq -\frac{1}{2}$. Hence,
\begin{equation*}
    \ell(\theta_{k+1}) \leq \ell(\theta_{k}) 
    - \frac{\alpha}{4} \left\| \nabla_{\theta} \ell(\theta_k) \right\|^2 + \frac{3\sigma^2 L^2}{2} \frac{\alpha^3 N}{B^2} \,.
\end{equation*}
Rearranging the above inequality, we see that
\begin{flalign*}
    \frac{1}{K} \sum_{k=0}^{K-1} \left\| \nabla_{\theta} \ell(\theta_k) \right\|^2 
    & \leq \frac{4}{\alpha K} \sum_{k=0}^{K-1} [\ell(\theta_k) - \ell(\theta_{k+1})] + 6\sigma^2 L^2 \alpha^2 \frac{N}{B^2} \\
    & = \frac{4}{\alpha K} [\ell(\theta_0) - \ell(\theta_K)] + 6\sigma^2 L^2 \alpha^2 \frac{N}{B^2} \\
    & \leq \frac{4}{\alpha K} [\ell(\theta_0) - \ell_*] + 6\sigma^2 L^2 \alpha^2 \frac{N}{B^2} \,.
\end{flalign*}
The final result is obtained by taking the expectation. 
\end{proof}

To conclude this subsection, we emphasize that the result of the above theorem only estimates the accuracy of each individual optimization step and does not address how the resulting optimization error propagates through the overall JKO scheme. Quantifying such propagation is highly nontrivial and lies beyond the scope of the present paper. Nevertheless, we note a recent preprint that addresses this issue for classical Wasserstein gradient flows \cite[Theorem~1.2]{dimarino2025}. Extending this analysis to the Landau case considered here remains nontrivial and constitutes an interesting direction for future work.

%%%%%%%%%%%%%%%%%%%%%%%%%%%
\subsection{Random batch method} \label{sec:rbm}
Utilizing the update equations in \eqref{update_1} with full particles can be computationally expensive when the particle number gets large or in high dimensions. The random batch method \cite{JIN2020108877, Carrillo_Jin_Tang_2022} is designed to address this challenge by restricting particle collisions to small, randomly selected batches. This is what we will adopt here. 

Given $N$ particles, we randomly divide them into $\frac{N}{B'}$ batches $C_q$, $q=1,\ldots,\frac{N}{B'}$, with a batch size $B' \ll N$. For the $i$--th particle in batch $C_q$, we update it using the following equation:
\begin{equation}\label{rbm}
    \begin{split}
        & z_i^{k+1} = z_i^{k} - \tau \frac{1}{B'} \sum_{j \in C_q} A(z_i^{k} - z_j^{k}) [s_\theta(\tau_k, z_i^{k}) - s_\theta(\tau_k, z_j^{k})] \,, \\
        & h_{i}^{k+1} = h_{i}^{k} - \tau \frac{1}{B'} \sum_{j \in C_q} \big\{ A(z_i^{k} - z_j^{k}) : \nabla s_\theta(\tau_k, z_i^{k}) - (d - 1) |z_i^{k} - z_j^{k}|^{\gamma} (z_i^{k} - z_j^{k}) \cdot [s_\theta(\tau_k, z_i^{k}) - s_\theta(\tau_k, z_j^{k})] \big\} \,, \\
        & z_i^0 = v_i^n \,,~~ h_{i}^{0} = 0 \,.
    \end{split}
\end{equation}
Since each particle interacts only within its own batch, the complexity of updating particles reduces from $O(N^2)$ to $O(B'N)$ per step. This approach accelerates the particle updates and preserves all physical quantities without significantly impacting the accuracy, as demonstrated in \cite[Proposition 3.1 and 3.3]{Carrillo_Jin_Tang_2022} and \cite[Theorem 3.1]{JIN2020108877}.

%%%%%%%%%%%%%%%%%%%%%%%%%%%%%%%%%%%%%%%%%
\section{Numerical examples}\label{sec5}
In this section, we present several numerical examples using the proposed JKO-based particle method for both the Maxwellian and Coulomb cases. To demonstrate the superior stability of the proposed method, we compare it with the score-based approach \cite{HUANG2025114053}. In each case, we fine-tune optimization hyperparameters such as the learning rate and the number of epochs to ensure that each method performs optimally, allowing for a fair comparison. Notably, we observe that when the collision effect is strong, the score-based approach is highly sensitive to the choice of learning rate and iteration count, whereas the JKO approach is much more robust to these parameter choices. 

As mentioned earlier in \eqref{opt_u1}, the optimal vector field $s(1, v)$ takes the form of the score function $\nabla \log f(1, v)$, with an additional term in $\operatorname{span} \{ \mathbf{1}, v \}$. This additional term vanishes when multiplied by $A(v-v_*)$. Based on this observation, the architecture of the neural network $s_\theta$ is similar to that used in the score-based particle method. Specifically, $s_\theta$ is parameterized as a fully connected neural network with 3 hidden layers, each containing 32 neurons and utilizing the \texttt{swish} activation function. We also note that the initialization of the neural network significantly impacts training efficiency. Our best practice protocol is as follows: for the first step, we initialize the biases to zero and set the weights using a truncated normal distribution with a variance of 1/\texttt{fan\_in}. Then we use the parameters from the previous time step to initialize the network at the next time step. The optimizer \texttt{AdamW} \cite{loshchilov2018decoupled} is used for all tests. 

All particles are initially sampled from the initial data using acceptance-rejection sampling,  and the random batch method \eqref{rbm} is employed for particle updates. To visualize particle density and compare it with the reference solution, we either use kernel density estimation with a Gaussian kernel $\psi_{\epsilon}$: 
\begin{equation} \label{kde}
   f_{\text{kde}}^n(v) := \frac{1}{N}\sum_{i=1}^N \psi_{\epsilon}(v-v_i^n) \,,~ \psi_{\epsilon}(v) = \frac{1}{(2\pi \epsilon^2)^{d/2}} \exp\left(-\frac{|v|^2}{2\epsilon^2}\right);
\end{equation}
or through the density update equation \eqref{update_2}, which enables us to obtain the density from those evolved particles directly. All experiments are implemented in PyTorch and executed on an Nvidia A40 GPU at the Minnesota Supercomputer Institute.

%%%%%%%%%%%%%%%%%%%%%%%%%%%
\subsection{BKW solution for Maxwellian molecules}
We first evaluate the accuracy of our method using the BKW solution to the Landau equation (see \cite[Appendix A]{carrillo2020particle}), which has an analytical expression and thus serves as an excellent benchmark for accuracy verification. Specifically, consider the collision kernel $A(z)=(|z|^2I_d - z \otimes z)$, then the BKW solution takes the form: 
\begin{equation} \label{BKW}
    f_{\text{bkw}}(t, v) = \frac{1}{(2\pi K)^{d/2}} \exp\left( -\frac{|v|^2}{2K} \right) \left( \frac{(d+2)K-d}{2K} \!+\! \frac{1-K}{2K^2}|v|^2 \right) \,,~ 
    K = 1 - D\exp\left(-2C_\gamma (d-1)t \right) \,.
\end{equation}

We begin with dimension $d=2$ and set $D = \frac{1}{2}$. We consider a weakly collisional regime with collision strength \(C_\gamma = \frac{1}{16}\) and time step of \(\Delta t = 0.01\). The total number of particles is $N=160^2$, and the batch size used for stochastic optimization and random batch updates is $1280$. We apply the JKO–1 and JKO–2 methods with the forward Euler solver (FE) to the same problem and examine how the number of inner time steps $N_{\tau}$ affects the overall accuracy. For both JKO--1 and JKO--2, we use a learning rate \(\alpha = 1 \times 10^{-2}\) and train for 50 epochs for the first JKO step and $5$ epochs for subsequent JKO steps.

The left and center plots of Fig.~\ref{fig_eg1_macro_weak} illustrate the conservation and the entropy dissipation properties of our solver. On the left, the kinetic energy is conserved with only a small error, and JKO--2 shows a smaller conservation error than JKO--1 due to its more accurate inner-time discretization. In the center, the entropy closely matches the analytical value for both JKO--1 and JKO--2. Based on this observation, the choice of inner time discretization has little impact on accuracy in the weakly collisional regime. Thus, we use the simplest JKO--1 (FE) method for all subsequent weakly collisional tests.

We further check the convergence in particle number $N$ by examining the following average $L^2$ error:
\begin{equation*}
    e_N = \sqrt{\frac{1}{J} \sum_{j=1}^J \left| H^N(f^{n}) - H(f_{\text{bkw}}(t^n)) \right|^2 } \,.
\end{equation*}
We compute $e_N$ over $J=10$ runs for each value of $N$, using different random seeds in each run to ensure independence. The right plot of Fig.~\ref{fig_eg1_macro_weak} shows the expected Monte Carlo convergence rate with respect to the particle number.

\begin{figure}[htbp]
    \centerline{\includegraphics[scale=0.38]{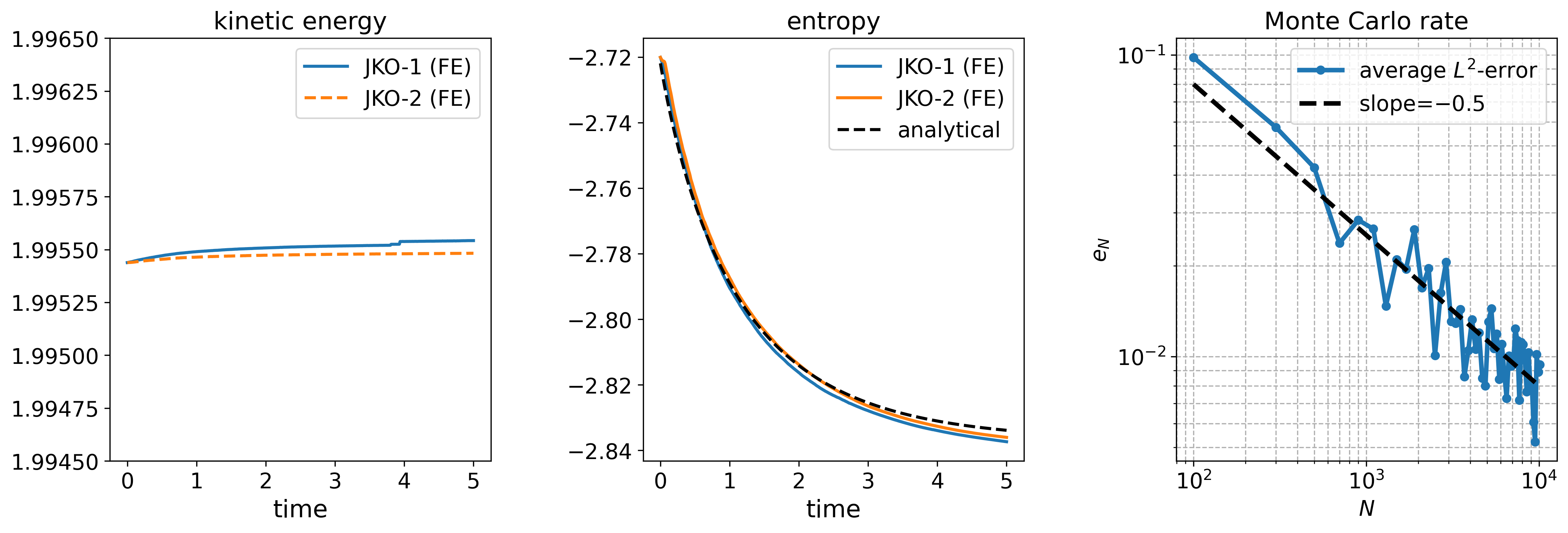}}
    \caption{Time evolution of macroscopic physical quantities for a 2D BKW solution with weak collision strength $C_\gamma=\frac{1}{16}$. Left: time evolution of the kinetic energy. Center: time evolution of the entropy. Right: Monte Carlo convergence rate with respect to particle number.}
    \label{fig_eg1_macro_weak}
\end{figure}

Next, we compare the stability of the proposed JKO-based particle method with its counterpart, the score-based particle method, in the strongly collisional regime, characterized by a larger $C_\gamma = 10$. In this setting, the solution converges to equilibrium more quickly, as evidenced by the rapid convergence of $K$ to 1 in \eqref{BKW}. To resolve this fast timescale, explicit schemes require a sufficiently small time step $\Delta t$. This is demonstrated in Fig.~\ref{fig_eg1_macro_strong}, where the score-based approach encounters instability even with a small time step $\Delta t = 0.01$, whereas the JKO--3 approach with the RK4 solver remains stable at a much larger step $\Delta t = 0.1$. The JKO--1 approach with the FE solver performs somewhat in between: it is stable at $\Delta t = 0.01$, showing better stability than the score-based approach thanks to its built-in entropy dissipation, but it becomes unstable at $\Delta t = 0.1$, showing reduced stability relative to the JKO--3 approach due to the larger inner time discretization error. 

\begin{figure}[htbp]
    \centerline{\includegraphics[scale=0.38]{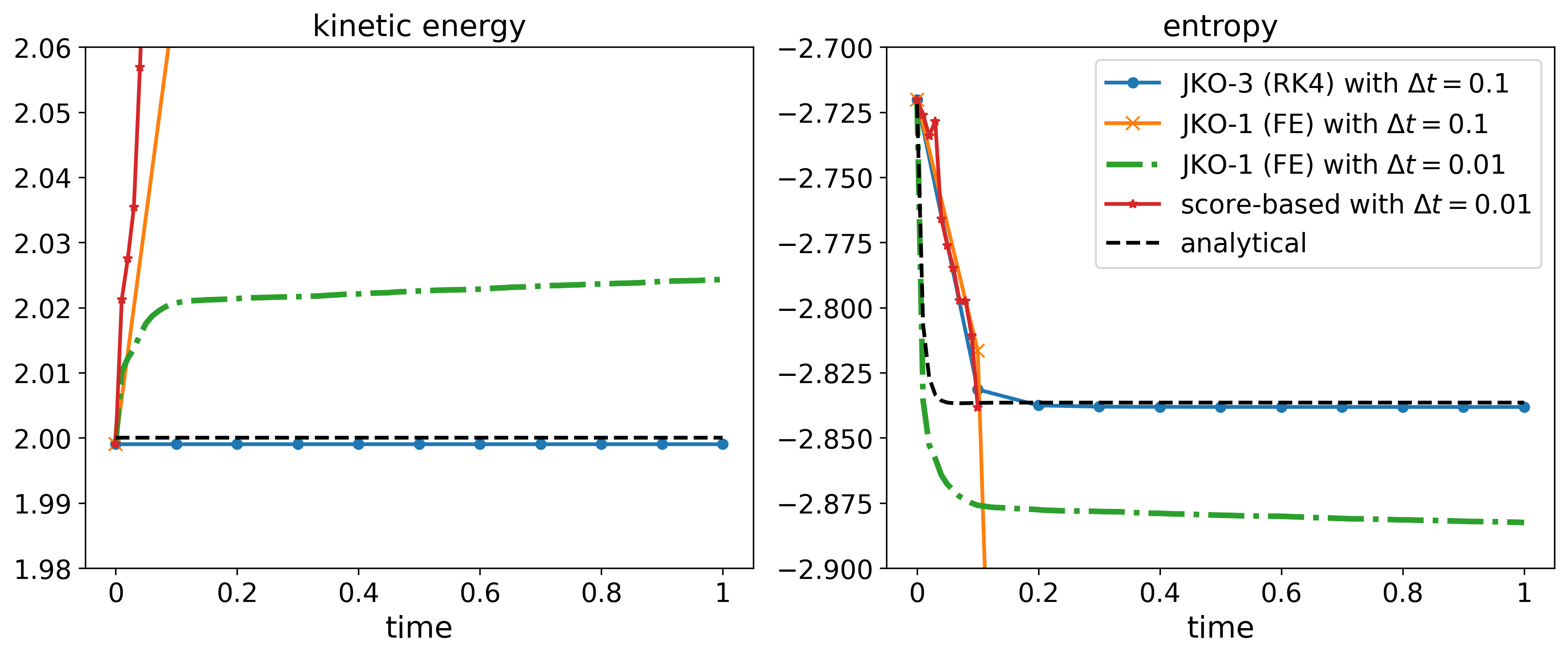}}
    \caption{Comparison between JKO--3 (RK4), JKO--1 (FE), and score-based particle methods with varying time step sizes for a 2D BKW solution with strong collision strength $C_\gamma=10$. Left: time evolution of the kinetic energy. Right: time evolution of the entropy.}
    \label{fig_eg1_macro_strong}
\end{figure}

We further extend the comparison to dimension $d=3$ with strong collisions. We set $D=1$, collision strength $C_\gamma=3$, and the total number of particles to $N=30^3$, with time steps $\Delta t = 0.01$ or $0.1$. Fig. \ref{fig_eg2_macro} again verifies the discrepancy in stability of the three methods, with the JKO--3 (RK4) method being the most stable, followed by the JKO--1 (FE) method, and the score-based particle method with the least stability.

\begin{figure}[htbp]
    \centerline{\includegraphics[scale=0.38]{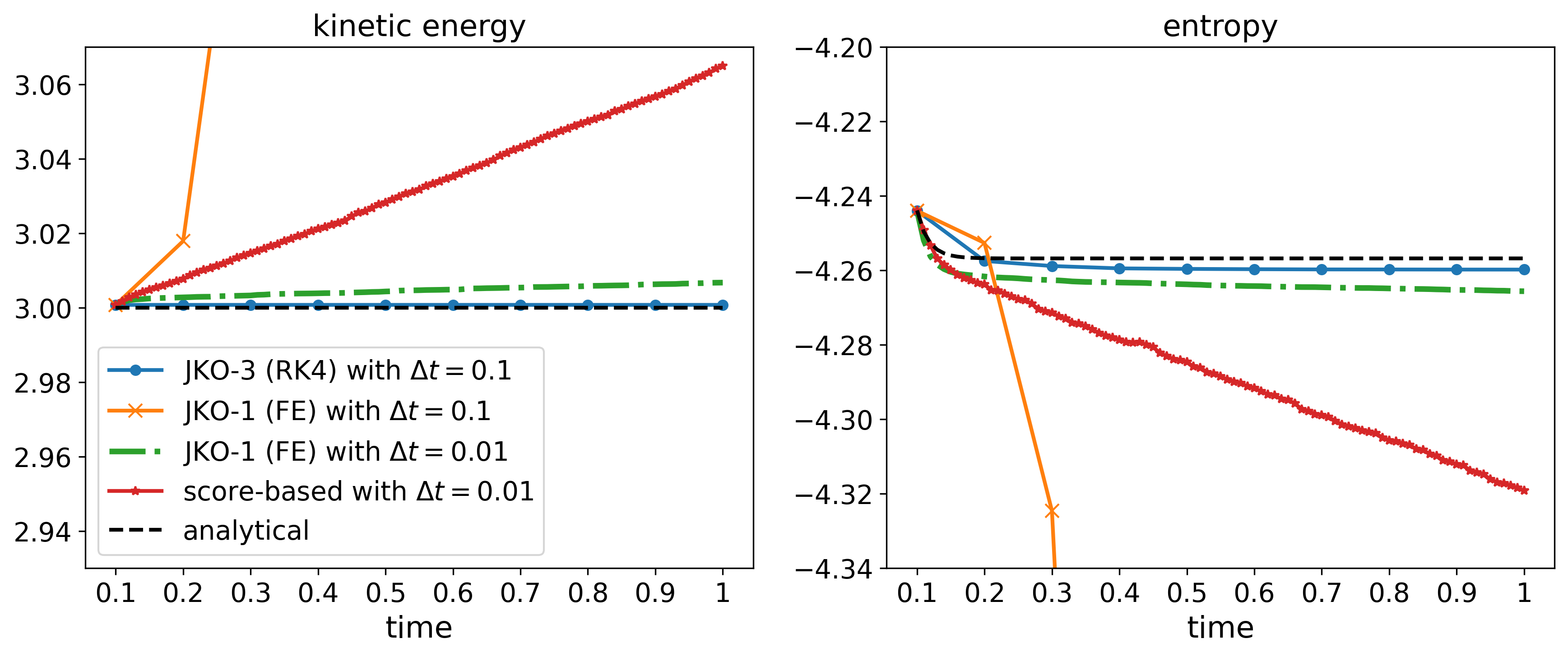}}
    \caption{Comparison between JKO--3 (RK4), JKO--1 (FE), and score-based particle methods with varying time step sizes for a 3D BKW solution with strong collision strength $C_\gamma=3$. Left: time evolution of the kinetic energy. Right: time evolution of the entropy.}
    \label{fig_eg2_macro}
\end{figure}

%%%%%%%%%%%%%%%%%%%%%%%%%%%
\subsection{2D bi-Maxwellian example with Coulomb interaction}
Now we consider a more physically relevant interaction: the Coulomb interaction. First, let us examine the two-dimensional case, where the collision kernel is
$A(z)=\frac{1}{|z|^3}(|z|^2I_2 - z \otimes z)$. We choose an initial distribution given by a bi-Maxwellian:
\begin{equation*}
    f^0(v) = \frac{1}{4\pi} \left\{ \exp\left( -\frac{|v-v_1|^2}{2} \right) + \exp\left( -\frac{|v-v_2|^2}{2} \right) \right\} ,~ v_1=(-2,1) ,~ v_2=(0,-1) \,.
\end{equation*}

In this experiment, the numerical parameters are set as follows: the collision strength is $C_\gamma=\frac{1}{16}$, the time step is $\Delta t=0.1$, the total number of particles is $N=120^2$, and the batch size is $1800$. For the JKO first step, we use a learning rate \(\alpha = 1 \times 10^{-2}\) and train for 50 epochs. For subsequent JKO steps, we adjust the learning rate to \(\alpha = 5 \times 10^{-4}\) and train for 5 epochs.

Since there is no analytical solution for this example, we compare our reconstructed solution from kernel density estimation \eqref{kde} with the blob solution obtained using the deterministic particle method from \cite{carrillo2020particle}, with the same number of particles. We set the computational domain to $[-10, 10]^2$ and use a Gaussian kernel with bandwidth $\epsilon = 0.3$ for density computation. The similarity between the reconstructed solution and the blob solution, as shown in Fig. \ref{fig_eg3_kde}, indicates that the JKO-based particle method is effective for Coulomb interactions. Fig. \ref{fig_eg3_macro} demonstrates that our method conserves energy and maintains entropy dissipation reasonably well. Additionally, Fig. \ref{fig_eg3_density} depicts the evolution of the particle solution obtained through the density formula \eqref{update_2} from $t=0$ to $t=160$, showing the transition from the initial bi-Maxwellian distribution to the equilibrium Maxwellian distribution.

\begin{figure}[htbp]
    \centerline{\includegraphics[scale=0.38]{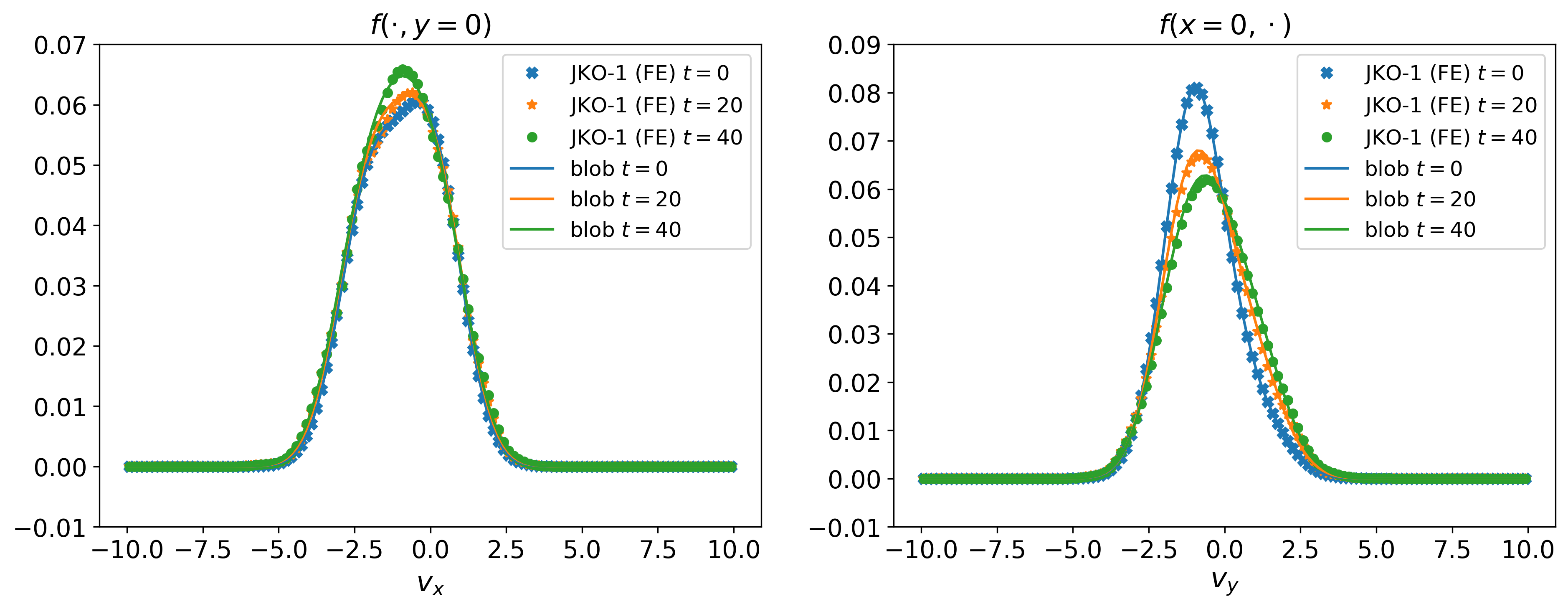}}
    \caption{Slice plots of the reconstructed and blob solution for an example of 2D Coulomb interaction at $t=0$, $20$, and $40$. Left: $f(\cdot, y=0)$. Right: $f(x=0, \cdot)$. }
    \label{fig_eg3_kde}
\end{figure}
\begin{figure}[htbp]
    \centerline{\includegraphics[scale=0.38]{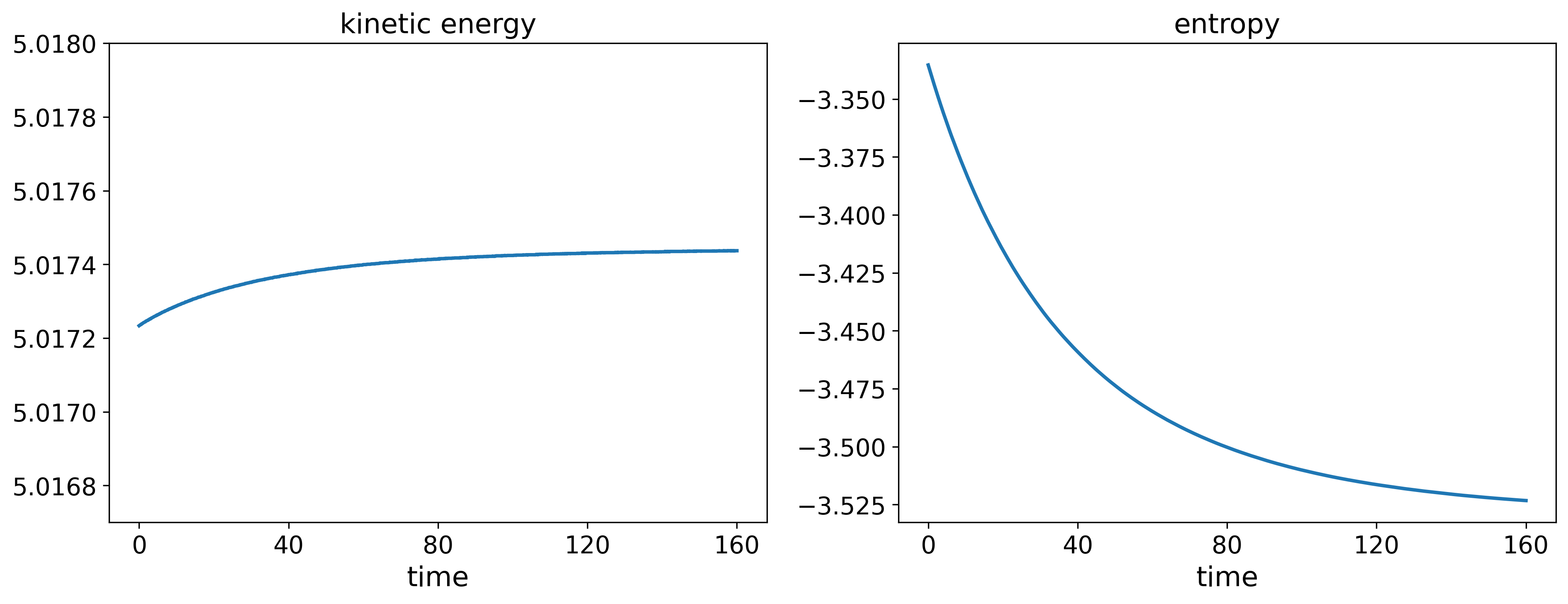}}
    \caption{Time evolution of macroscopic physical quantities for an example of 2D Coulomb interaction. Left: time evolution of the kinetic energy. Right: time evolution of the entropy.}
    \label{fig_eg3_macro}
\end{figure}
\begin{figure}[htbp]
    \centerline{\includegraphics[scale=0.38]{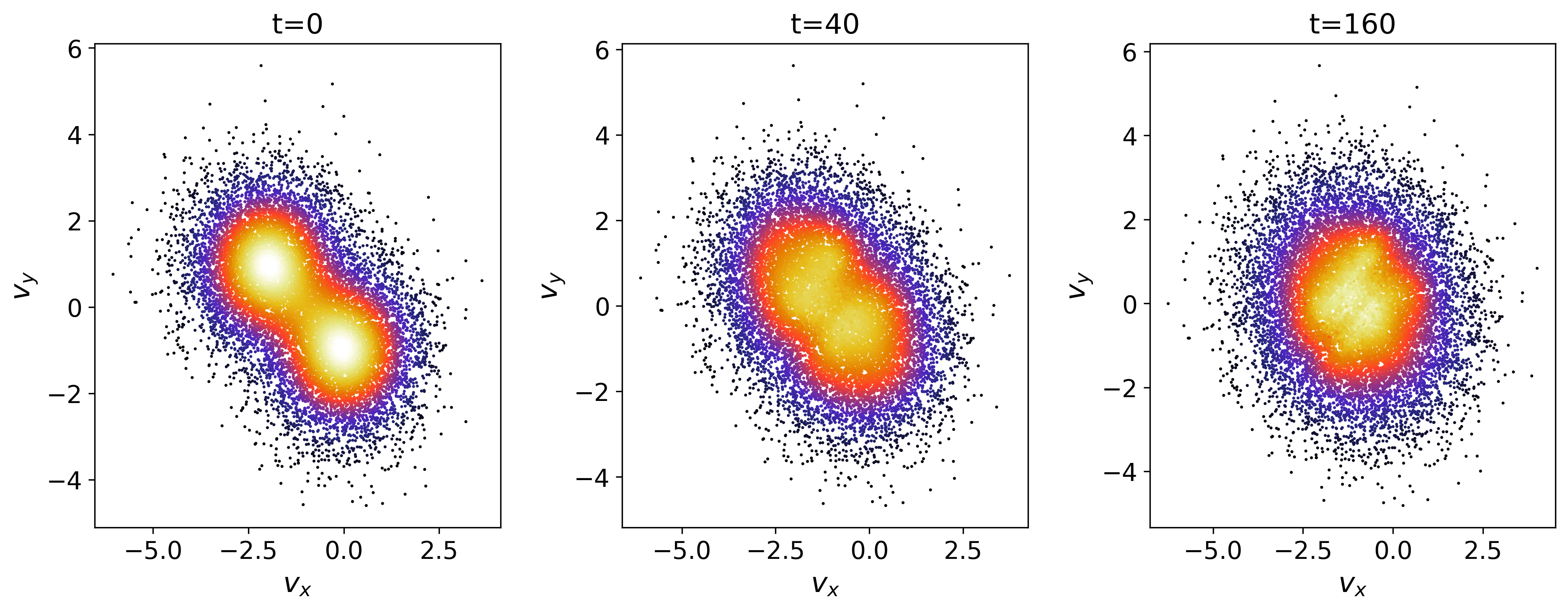}}
    \caption{Scatter plots of particle solution for an example of 2D Coulomb interaction at $t=0$, $40$, and $160$.}
    \label{fig_eg3_density}
\end{figure}

%%%%%%%%%%%%%%%%%%%%%%%%%%%
\subsection{3D Rosenbluth problem with Coulomb interaction}\label{sec:5.3}
Here, we extend our consideration of Coulomb interactions to three dimensions, using the collision kernel given by: $A(z)=\frac{1}{|z|^3}(|z|^2I_3 - z \otimes z)$. The initial condition is set to
\begin{equation*}
    f^0(v) = \frac{1}{S^2} \exp\left( -S\frac{(|v|-\sigma)^2}{\sigma^2} \right) \,,~ \sigma=0.3 \,,~ S=10 \,.
\end{equation*}

We first consider a weakly collisional regime with collision strength \(C_\gamma = \frac{1}{4\pi}\). We use $N = 40^3$ particles, a time step of $\Delta t = 0.2$, and a batch size of $2560$. The learning rate is set to \(\alpha = 1 \times 10^{-2}\) and we train for 50 epochs for the first JKO step and 5 epochs for subsequent JKO steps.

Fig. \ref{fig_eg4_kde} shows the reconstructed density using kernel density estimation \eqref{kde} with $\epsilon=0.035$. The result is in good agreement with the blob and spectral solution presented in \cite{carrillo2020particle}. We also test the computational efficiency of the proposed method in comparison with two other siblings: the blob method \cite{carrillo2020particle} and the score-based particle method \cite{HUANG2025114053}. In the weakly collisional regime, the blob, score-based, and JKO-based particle methods all admit the same time step size. Thus, the main difference in efficiency comes from how each method computes the score, i.e., $\nabla\log f$. Specifically, the blob method uses kernel density estimation, whereas both the score-based and the JKO-based methods learn the score through optimization, with the former using a score-matching technique, and the latter minimizing a loss function derived from the JKO scheme. As shown in the left plot of Fig. \ref{fig_eg4_time}, both the score-based method and the JKO--1 (FE) approach scale as $O(N)$, with the score-based approach being faster. This is expected, as both approaches involve learning the score through neural network training, but the score-based method has a simpler loss function. However, the JKO--1 (FE) approach offers additional benefits not shared with the score-based approach, such as better stability and exact entropy decay. The right plot of Fig.~\ref{fig_eg4_time} demonstrates the efficiency of the random batch method \eqref{rbm}, clearly showing that it reduces the computational cost from $O(N^2)$ to $O(N)$. It is important to note that this reduction is particularly significant in higher dimensions, such as in the spatially inhomogeneous Landau equation.

\begin{figure}[htbp]
    \centerline{\includegraphics[scale=0.33]{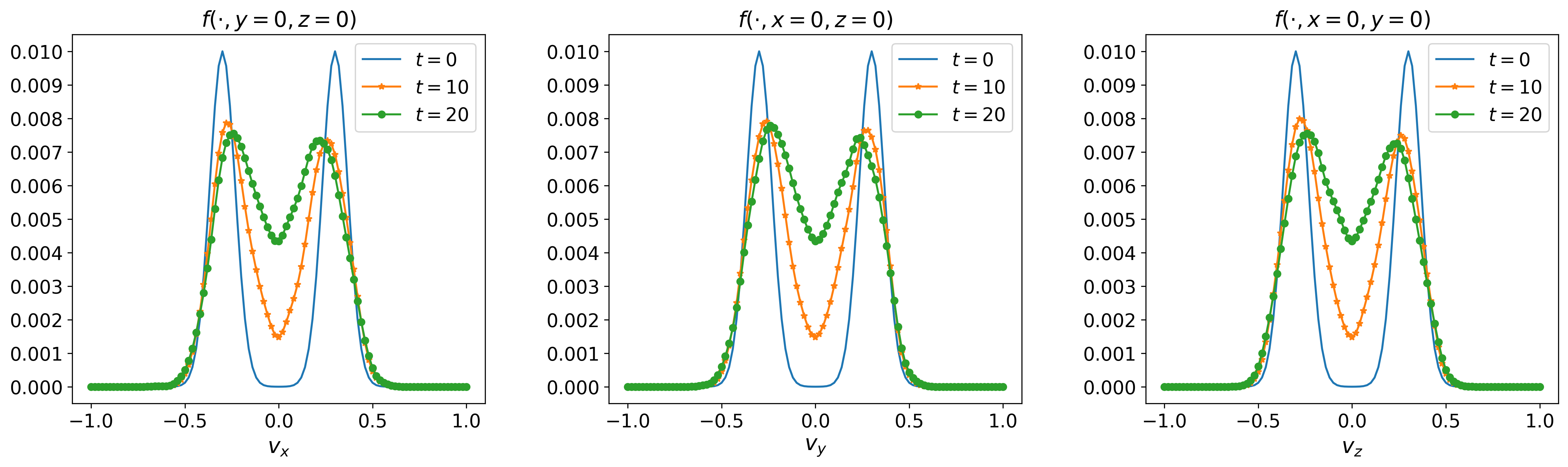}}
    \caption{Slice plots of the reconstructed solution for a 3D Rosenbluth problem with Coulomb interaction in weakly collisional regime $(C_\gamma=\frac{1}{4\pi})$ at $t=0$, $10$, and $20$.}
    \label{fig_eg4_kde}
\end{figure}

\begin{figure}[htbp]
    \centerline{\includegraphics[scale=0.38]{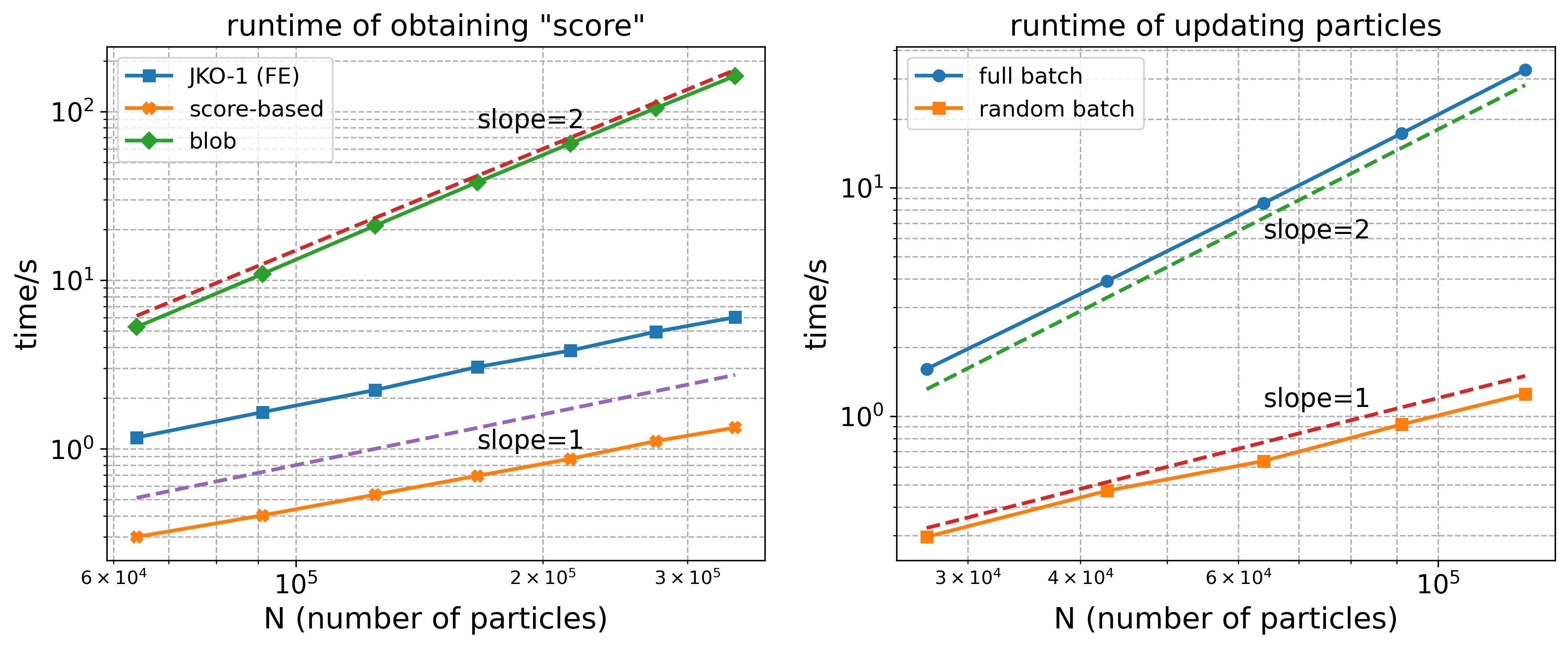}}
    \caption{Comparison of the runtime (in seconds) for a 3D Rosenbluth problem with Coulomb interaction in weakly collisional regime $(C_\gamma=\frac{1}{4\pi})$ on GPU. Left: a comparison for obtaining the ``score'' function using the blob, score-based, and JKO--1 (FE) particle methods. Right: a comparison for updating particles using full batch and random batch methods.}
    \label{fig_eg4_time}
\end{figure}

To demonstrate the stability of the JKO approach in the Coulomb interaction, we consider a strong collisional regime with $C_\gamma = 100$ and apply the JKO--3 (RK4) approach using a relatively large time step of $\Delta t = 1$ and $N=25^3$ particles. The results, shown in Fig. \ref{fig_eg4_macro}, indicate that our method remains stable despite the large time step and that the entropy approaches equilibrium effectively. In contrast, the score-based particle method requires a much smaller time step $\Delta t = 0.01$ to remain stable and still fails to stay near equilibrium. Moreover, the JKO--3 method finishes in only 243 seconds, while the score-based particle method takes 1014 seconds, clearly showing that the JKO-based approach is significantly more efficient in stiff regimes.

\begin{figure}[htbp]
    \centerline{\includegraphics[scale=0.38]{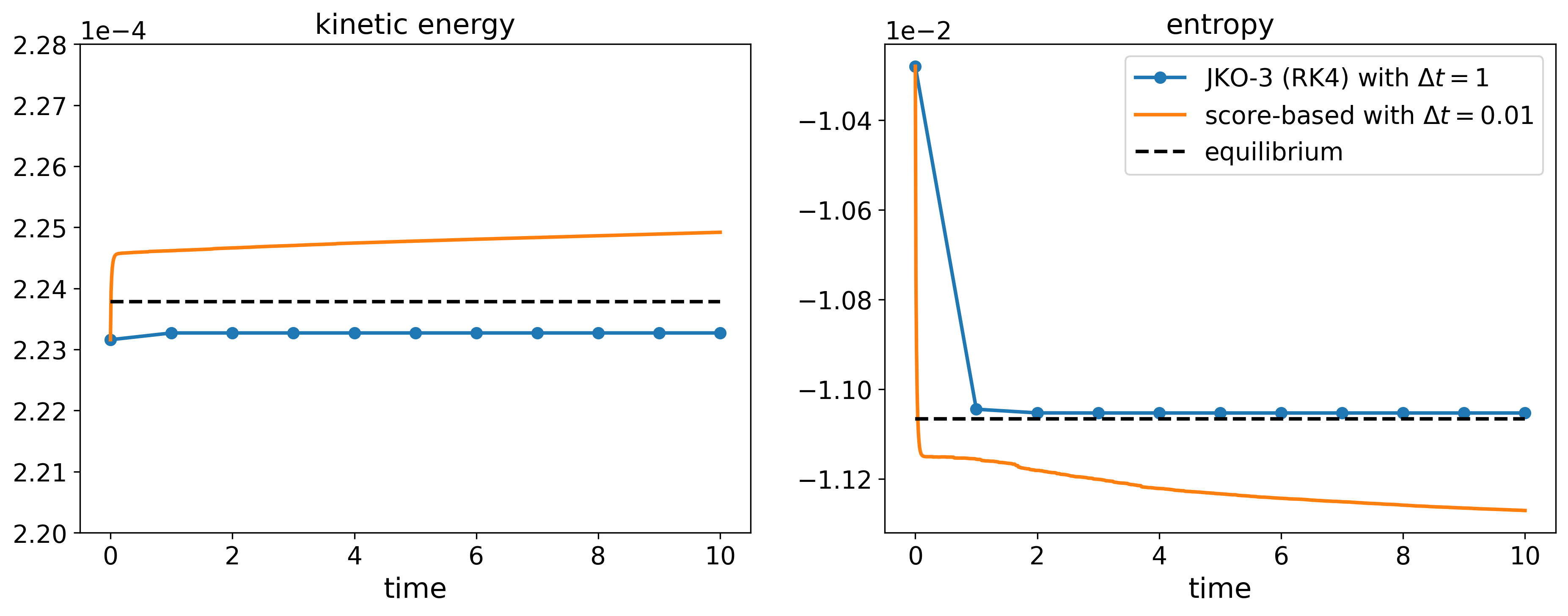}}
    \caption{Comparison between JKO--3 (RK4) and score-based particle methods for a 3D Rosenbluth problem with Coulomb interaction in strongly collisional regime $(C_\gamma=100)$. The time step for JKO--3 (RK4) and score-based particle method are $\Delta t=1$ and $\Delta t=0.01$, respectively. Left: time evolution of the kinetic energy. Right: time evolution of the entropy.}
    \label{fig_eg4_macro}
\end{figure}

%%%%%%%%%%%%%%%%%%%%%%%%%%%
\subsection{10D anisotropic solution with Maxwellian molecules}
In this test, we aim to demonstrate the potential of our proposed method in managing high-dimensional problems, leveraging the use of particles and neural networks for approximation. The example we use is adapted from \cite[Example 5.3]{ilin2024}.

Consider the Maxwellian collision kernel $A(z)=|z|^2I_d - z \otimes z$ with collision strength $C_\gamma=1$, and set the initial distribution as a normal distribution with zero mean and covariance $P_{i,j} = \delta_{i,j} p_i \,, p_1=1.8 \,, p_2=0.2 \,, p_i=1 \,, i=3,\ldots,d$. In this case, there is no explicit formula for the density, but there is one for the covariance matrix (stress tensor), $P(t) = \int_{\R^d} v \otimes v f(t, v) \rd v$, as shown in \cite{villani1998a}: 
\begin{equation*}
    P_{i,j}(t) = P_{i,j}(\infty) - (P_{i,j}(\infty)-P_{i,j}(0)) e^{-4dt} \,,~ P_{i,j}(\infty) = \frac{E}{d} \delta_{i,j} \,,~ E = \int_{\R^d} |v|^2 f \rd v \,.
\end{equation*}

In our test, we consider a dimension of $d=10$ and set the time step to $\Delta t=0.002$, using a total of $N=25600$ particles and a batch size of $1280$. The learning rate is set to \(\alpha = 1 \times 10^{-2}\) and we train for $30$ epochs for the first 20 JKO steps. For subsequent JKO steps, the learning rate is reduced to \(\alpha = 1 \times 10^{-3}\) and we train for $5$ epochs. 

To examine the accuracy of our method, we compare the numerical covariance using particles
$P^{N}_{i,j}(t^n) = \frac{1}{N} \sum_{k=1}^N (v^n_k)_i (v^n_k)_j$ with the analytical solution. Fig.~\ref{fig_eg5_cov} plots the time evolution of selected diagonal elements of the numerical covariance, which closely matches the analytical covariance. Moreover, we compute the Frobenius norm between the numerical and analytical covariances, 
$\sqrt{\sum_{i,j=1}^N (P^{N}_{i,j} - P_{i,j})^2}$, and the results are presented in Fig. \ref{fig_eg5_cov_err}. On the left, we show the time evolution of the Frobenius error with a fixed particle number $N=25600$, which remains small throughout the experiment. On the right, we track the convergence in particle number $N$ by computing the following average Frobenius error:
\begin{equation*}
    c_N = \sqrt{\frac{1}{J} \sum_{q=1}^J \sum_{i,j=1}^N |P^{N}_{i,j} - P_{i,j}|^2 } \,.
\end{equation*}
We compute $c_N$ over $J=10$ runs for each $N$. We again observe the expected Monte Carlo convergence rate of $-\frac{1}{2}$, which is independent of the dimension.

\begin{figure}[htbp]
    \centerline{\includegraphics[scale=0.38]{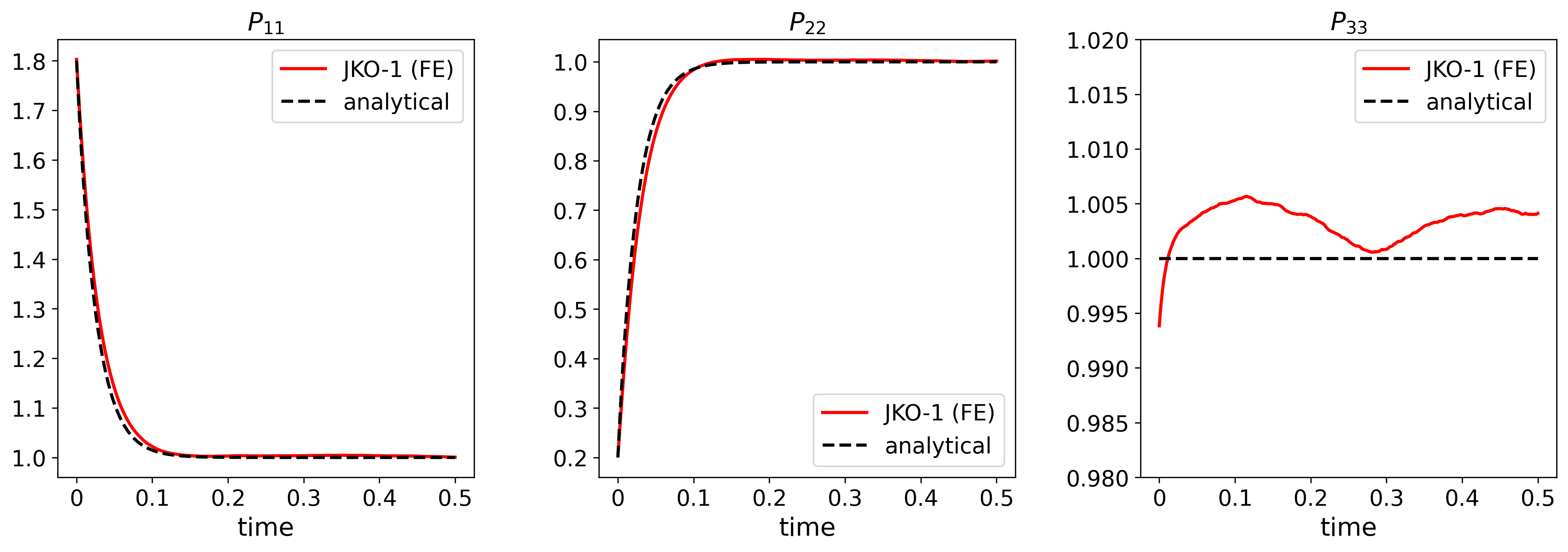}}
    \caption{Time evolution of the first to third diagonal elements of the covariance matrix for the 10-D example.}
    \label{fig_eg5_cov}
\end{figure}
\begin{figure}[htbp]
    \centerline{\includegraphics[scale=0.38]{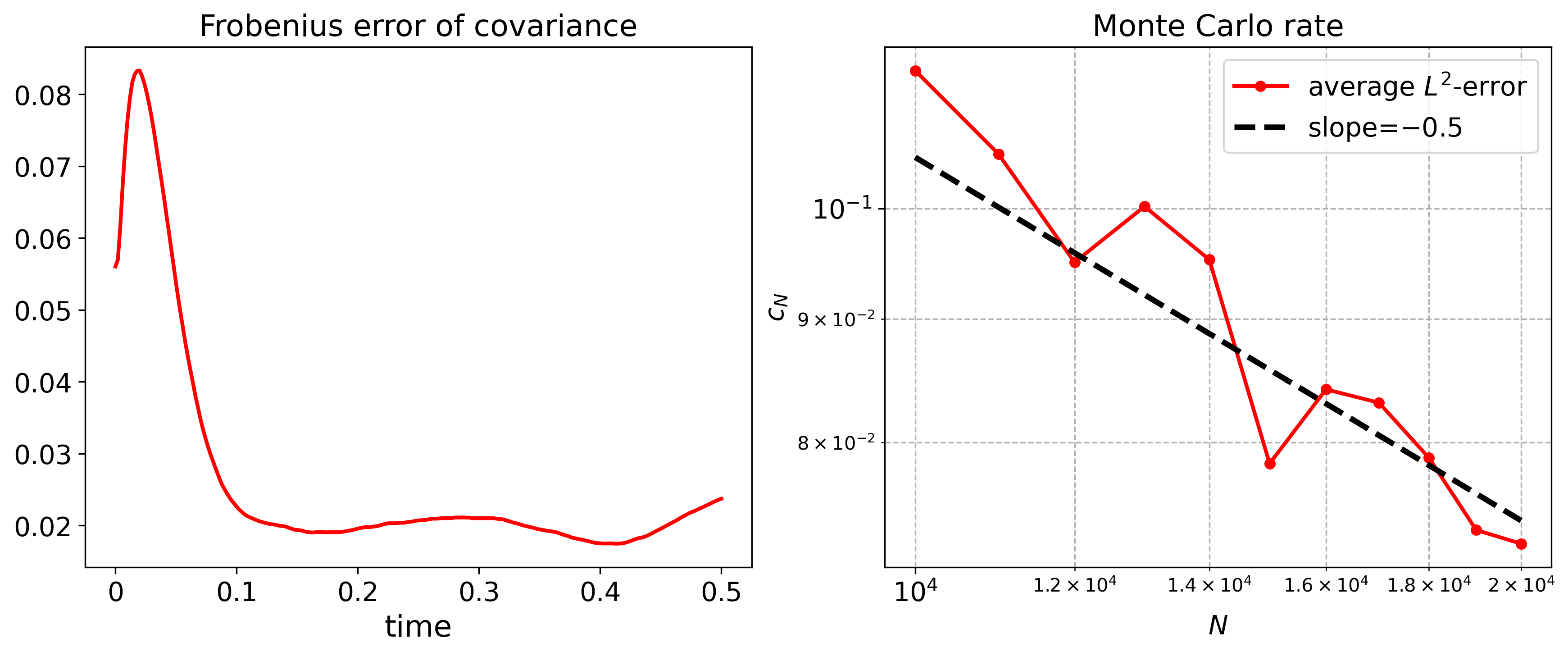}}
    \caption{Error in covariance for the 10-D example. Left: Frobenius error versus time with fixed particle number $N=25600$. Right: rate of convergence of the Frobenius error with respect to particle numbers.}
    \label{fig_eg5_cov_err}
\end{figure}

\section{Conclusion and discussion}
In this paper, we design a neural network-based variational particle method for solving the homogeneous Landau equation, leveraging the underlying gradient flow structure. To address the complexity of the Landau metric, our main contributions include reformulating it into a form that can be easily represented by particles and recognizing that the unknown part is simply a score function of the corresponding density. These key insights provide essential guidance for neural network approximation and initialization. It is important to note that the objective function, when represented using particles, is in the form of a double summation, which motivates us to design a tailored mini-batch stochastic optimization method. The particle update is further accelerated using the random batch method. 

We have demonstrated that, compared to recent score-based particle methods for the Landau equation \cite{HUANG2025114053, ilin2024}, our approach involves a more complex objective function. However, this complexity is significantly mitigated by the use of stochastic methods. Additionally, our method offers unique advantages that the score-based particle method does not, including exact entropy decay and unconditional stability. When combined with the particle-in-cell method \cite{filbet2003numerical, chen2011energy, bailo2024collisional}, these features make our approach promising for large-scale long-time plasma simulations.

%%%%%%%%%%%%%%%%%%%%%%%%%%%%%
\appendix
\section{Review of general gradient flow theory}\label{app_def}
We recall some important definitions of gradient flow on metric spaces that can be found in Chapter 1 of \cite{Ambrosio2005GradientFI}. Throughout this section, we consider the complete metric space $(X, d)$ and the proper extended real functional $\varphi: X \to (-\infty, +\infty]$.

\begin{definition}[Absolutely continuous curves]
    A curve $\mu: (a,b) \to X$ is said to be an absolutely continuous curve if there exists $m \in L^2(a,b)$ such that $d(\mu_s, \mu_t) \leq \int_s^t m(r) \rd r$, for $\forall a < s \leq t < b$.
\end{definition}

\begin{definition}[Metric derivative]
    For any absolutely continuous curve $\mu: (a,b) \to X$, the limit
    $|\mu'|(t) := \lim_{s \to t} \frac{d(\mu(s), \mu(t))}{|s-t|}$,
    exists for a.e. $t \in (a,b)$, and we define the limit as metric derivative of $\mu$ at $t$.
\end{definition}

\begin{definition}[Strong upper gradients]
    A function $g: X \to [0, +\infty]$ is a strong upper gradient for $\varphi$ if for every  absolutely continuous curve $\mu: (a,b) \to X$, the function $g \circ \mu$ is Borel and 
    $|\varphi(\mu_t) - \varphi(\mu_s)| \leq \int_s^t g(\mu_r) |\mu'|(r) \rd r$, for $\forall a < s \leq t < b$.
\end{definition}

\begin{definition}[Curves of maximal slope] \label{def-maxslope}
    An absolutely continuous curve $\mu: (a,b) \to X$ is said to be a curve of maximal slope for $\varphi$ with respect to its strong upper gradient $g$, if $\varphi \circ \mu$ is equal to a non-increasing map a.e. in $(a,b)$ and 
    \begin{equation*}
        \varphi(\mu_t) - \varphi(\mu_s) + \frac{1}{2} \int_s^t g(\mu_r)^2 \rd r + \frac{1}{2} \int_s^t |\mu'|^2(r) \rd r \leq 0 \,,~ \forall a < s \leq t < b \,.
    \end{equation*}
\end{definition}
Note that the above inequality is indeed equality by applying Young's inequality in the definition of the strong upper gradient.

\section{Lemmas for Theorem \ref{thm_opt_conv}}
\begin{lemma}\label{sgd_lemma1}
Suppose that assumptions (A.2) and (A.3) hold. Let learning rate $\alpha^{(k)}_q := \alpha_k \frac{B}{N} > 0$ for a given sequence $\{ \alpha_k \}$ such that $0 < \alpha_k \leq \frac{1}{L\sqrt{3}}$. Then we have
    \begin{equation*}
        \Delta := \sum_{q=1}^{\frac{N}{B}} \left\| \theta^{(k)}_{q-1} - \theta^{(k)}_0 \right\|^2
        \leq
        \frac{\alpha_k^2 N^2}{B^2} \left( (3M+2) \left\| \nabla_{\theta} \ell(\theta^{(k)}_0) \right\|^2 + 3\sigma^2 \right) \,.
    \end{equation*}
\end{lemma}
\begin{proof}
By definition, $\theta^{(k)}_q - \theta^{(k)}_0 = \frac{\alpha_k}{N B} \sum_{p=1}^{q} \sum_{i,j \in C_p} \nabla_{\theta} \ell_{i,j}(\theta^{(k)}_{p-1})$.
Then, by using Cauchy-Schwarz's inequality $\| \sum_{i=1}^n a_i \|^2 \leq n \sum_{i=1}^n \| a_i\|^2$ repeatedly, we have
\begin{flalign*}
    \| \theta^{(k)}_q - \theta^{(k)}_0 \|^2 
    & \leq \frac{3\alpha_k^2 q^2 B^2}{N^2} \bigg[ \| \frac{1}{qB^2} \sum_{p=1}^{q} \sum_{i,j \in C_p} \left( \nabla_{\theta} \ell_{i,j}(\theta^{(k)}_{p-1}) - \nabla_{\theta} \ell_{i,j}(\theta^{(k)}_{0}) \right) \|^2 + \\
    & \hspace{2cm}  \| \frac{1}{qB^2} \sum_{p=1}^{q} \sum_{i,j \in C_p} \left( \nabla_{\theta} \ell_{i,j}(\theta^{(k)}_{0}) - \nabla_{\theta} \ell(\theta^{(k)}_{0}) \right) \|^2 + \| \nabla_{\theta} \ell(\theta^{(k)}_{0}) \|^2 \bigg] \\
    & \leq \frac{3\alpha_k^2 q^2 B^2}{N^2} \bigg[ \frac{1}{qB^2} \sum_{p=1}^{q} \sum_{i,j \in C_p}  \| \nabla_{\theta} \ell_{i,j}(\theta^{(k)}_{p-1}) - \nabla_{\theta} \ell_{i,j}(\theta^{(k)}_{0}) \|^2 + \\
    & \hspace{2cm} \frac{1}{qB^2} \sum_{i,j =1}^N \| \nabla_{\theta} \ell_{i,j}(\theta^{(k)}_{0}) - \nabla_{\theta} \ell(\theta^{(k)}_{0}) \|^2 + \| \nabla_{\theta} \ell(\theta^{(k)}_{0}) \|^2 \bigg] \\
    & \leq \frac{3\alpha_k^2 q^2 B^2}{N^2} \left[ \frac{L^2}{q} \sum_{p=1}^{q} \| \theta^{(k)}_{p-1} - \theta^{(k)}_{0} \|^2 + \frac{N^2}{qB^2} \left( M \| \nabla_{\theta} \ell(\theta^{(k)}_{0}) \|^2 + \sigma^2 \right) + \| \nabla_{\theta} \ell(\theta^{(k)}_{0}) \|^2 \right] \\
    & \leq \frac{3\alpha_k^2 L^2 B^2 q}{N^2} \Delta + 3\alpha_k^2 \left[ q \left( M \| \nabla_{\theta} \ell(\theta^{(k)}_{0}) \|^2 + \sigma^2 \right) + \frac{q^2 B^2}{N^2} \| \nabla_{\theta} \ell(\theta^{(k)}_{0}) \|^2 \right] \,,
\end{flalign*}
where we have used the assumptions (A.2) and (A.3) in the second last inequality. Summing from $1$ to $\frac{N}{B}$,
\begin{flalign*}
    \Delta = \sum_{q=1}^{\frac{N}{B}} \| \theta^{(k)}_{q-1} - \theta^{(k)}_0 \|^2 
    \leq \frac{3\alpha_k^2 L^2}{2} \Delta + \frac{3\alpha_k^2 N^2}{2B^2} \left( M \left\| \nabla_{\theta} \ell(\theta^{(k)}_{0}) \right\|^2 + \sigma^2 \right) + \frac{\alpha_k^2 N}{B} \left\| \nabla_{\theta} \ell(\theta^{(k)}_{0}) \right\|^2 \,.
\end{flalign*}
Since $\alpha_k \leq \frac{1}{L\sqrt{3}}$, then $\frac{3\alpha_k^2 L^2}{2} \leq \frac{1}{2}$. Therefore, 
\begin{equation*}
    \Delta \leq \frac{\alpha_k^2 N^2}{B^2} \left( (3M+2) \left\| \nabla_{\theta} \ell(\theta^{(k)}_{0}) \right\|^2 + 3\sigma^2 \right) \,.
\end{equation*}
\end{proof}

\begin{lemma}\label{sgd_lemma2}
Suppose that assumption (A.2) holds. Let learning rate $\alpha^{(k)}_q := \alpha_k \frac{B}{N} > 0$ for a given sequence $\{ \alpha_k \}$ such that $0 < \alpha_k \leq \frac{1}{L}$. Then we have
    \begin{equation*}
        \ell(\theta_{k+1}) 
        \leq 
        \ell(\theta_{k}) + \frac{\alpha_k L^2 B}{2N} \sum_{q=1}^{\frac{N}{B}} \| \theta_{k} - \theta^{(k+1)}_{q-1} \|^2 - \frac{\alpha_k}{2} \| \nabla\ell(\theta_{k}) \|^2 \,.
    \end{equation*}
\end{lemma}
\begin{proof}
By assumption (A.2), $\ell$ is also $L$-smooth. Then we derive
\begin{flalign*}
    \ell(\theta_{k+1}) 
    & \leq \ell(\theta_{k}) + \langle \nabla\ell(\theta_{k}), \theta_{k+1} - \theta_{k} \rangle + \frac{L}{2} \left\| \theta_{k+1} - \theta_{k} \right\|^2 \\
    & = \ell(\theta_{k}) - \frac{\alpha_k}{NB} \langle \nabla\ell(\theta_{k}), \sum_{q=1}^{\frac{N}{B}} \sum_{i,j \in C_q} \nabla_{\theta} \ell_{i,j}(\theta^{(k+1)}_{q-1}) \rangle + \frac{L}{2} \frac{\alpha_k^2}{N^2 B^2} \| \sum_{q=1}^{\frac{N}{B}} \sum_{i,j \in C_q} \nabla_{\theta} \ell_{i,j}(\theta^{(k+1)}_{q-1}) \|^2 \\
    & = \ell(\theta_{k}) + \frac{\alpha_k}{2} \| \nabla\ell(\theta_{k}) - \frac{1}{BN} \sum_{q=1}^{\frac{N}{B}} \sum_{i,j \in C_q} \nabla_{\theta} \ell_{i,j}(\theta^{(k+1)}_{q-1}) \|^2 - \frac{\alpha_k}{2} \| \nabla\ell(\theta_{k}) \|^2 \\
    & \hspace{0.5cm} + \left( \frac{L}{2} \frac{\alpha_k^2}{N^2 B^2} - \frac{\alpha_k}{2N^2 B^2} \right) \| \sum_{q=1}^{\frac{N}{B}} \sum_{i,j \in C_q} \nabla_{\theta} \ell_{i,j}(\theta^{(k+1)}_{q-1}) \|^2 \,,
\end{flalign*}
where last equality uses the identity $ab = \frac{1}{2} (a^2 + b^2 - (a-b)^2)$.
Since $\alpha_k \leq \frac{1}{L}$, then we have $\frac{L}{2} \frac{\alpha_k^2}{N^2 B^2} - \frac{\alpha_k}{2N^2 B^2} \leq 0$. Thus,
\begin{flalign*}
    \ell(\theta_{k+1}) 
    & \leq \ell(\theta_{k}) 
    + \frac{\alpha_k}{2} \| \nabla\ell(\theta_{k}) - \frac{1}{BN} \sum_{q=1}^{\frac{N}{B}} \sum_{i,j \in C_q} \nabla_{\theta} \ell_{i,j}(\theta^{(k+1)}_{q-1}) \|^2 
    - \frac{\alpha_k}{2} \| \nabla\ell(\theta_{k}) \|^2 \\
    & \leq \ell(\theta_{k}) 
    + \frac{\alpha_k}{2} \frac{1}{BN} \sum_{q=1}^{\frac{N}{B}} \sum_{i,j \in C_q} \|  \nabla\ell(\theta_{k}) - \nabla_{\theta} \ell_{i,j}(\theta^{(k+1)}_{q-1}) \|^2 
    - \frac{\alpha_k}{2} \| \nabla\ell(\theta_{k}) \|^2 \\
    & \leq \ell(\theta_{k}) 
    + \frac{\alpha_k L^2 B}{2N} \sum_{q=1}^{\frac{N}{B}} \| \theta_{k} - \theta^{(k+1)}_{q-1} \|^2 
    - \frac{\alpha_k}{2} \| \nabla\ell(\theta_{k}) \|^2 \,.
\end{flalign*}
\end{proof}

\section*{Acknowledgments}
LW would like to thank Prof. Jose A. Carrillo and Dr. Jeremy Wu for inspiring discussion on the gradient flow formulation of the Landau equation. We also acknowledge ChatGPT for its assistance in refining the language and improving the clarity of the manuscript.

\section*{Funding}
This work is partially supported by NSF grant DMS-1846854, DMS-2513336, UMN DSI-SSG-4886888864, and the Simons Fellowship.

\bibliographystyle{plain}
\bibliography{ref.bib}

\end{document}